\documentclass[oneside,leqno,11pt]{paper}
\usepackage{amsmath,amssymb,latexsym}
\usepackage{pdfpages}
\usepackage[colorinlistoftodos]{todonotes}
\usepackage{cite}
\usepackage{multirow}
\usepackage{pgfplots}

\pgfplotsset{compat=1.7}

\usepackage{graphicx}
\usepackage[english]{babel}
\usepackage[utf8]{inputenc}
\usepackage{amsfonts}
\usepackage{hyperref}
\usepackage{enumerate}
\usepackage{bibentry}
\nobibliography*

\usepackage{graphicx}
\usepackage{float}

\makeatletter
\def\ball{{I\kern -.35em B}}
\def\tto{\rightrightarrows}
\def\bx{\bar x}
\def\by{\bar y}
\def\bz{\bar z}

\def\nn{neighborhood\,}
\def\nns{neighborhoods\,}

\def\dom{\mathop{\rm dom}\nolimits}

\def\gph{\mathop{\rm gph}\nolimits}

\def\dist{\mathop{\rm dist}\nolimits}

\newtheorem{proof}{Proof}

\newtheorem{definition}{Definition}[section]
\newtheorem{proposition}{Proposition}[section]
\newtheorem{theorem}{Theorem}[section]

\newtheorem{lemma}{Lemma}[section]
\newtheorem{remark}{Remark}[section]
\newtheorem{example1}{Example}[section]

\newtheorem{algorithm}[definition]{Algorithm}

\pgfdeclarepatternformonly[\LineSpace]{my north east lines}{\pgfqpoint{-1pt}{-1pt}}{\pgfqpoint{\LineSpace}{\LineSpace}}{\pgfqpoint{\LineSpace}{\LineSpace}}%
{
	\pgfsetcolor{\tikz@pattern@color}
	\pgfsetlinewidth{0.1pt}
	\pgfpathmoveto{\pgfqpoint{0pt}{0pt}}
	\pgfpathlineto{\pgfqpoint{\LineSpace + 0.1pt}{\LineSpace + 0.1pt}}
	\pgfusepath{stroke}
}

\pgfdeclarepatternformonly[\LineSpace]{my north west lines}{\pgfqpoint{-1pt}{-1pt}}{\pgfqpoint{\LineSpace}{\LineSpace}}{\pgfqpoint{\LineSpace}{\LineSpace}}%
{
	\pgfsetcolor{\tikz@pattern@color}
	\pgfsetlinewidth{0.1pt}
	\pgfpathmoveto{\pgfqpoint{0pt}{\LineSpace}}
	\pgfpathlineto{\pgfqpoint{\LineSpace + 0.1pt}{-0.1pt}}
	\pgfusepath{stroke}
}
\makeatother

\newdimen\LineSpace
\tikzset{
	line space/.code={\LineSpace=#1},
	line space=5pt
}
\begin{document}

\title{On semismooth$^*$ path-following method and uniformity of strong metric subregularity at/around the reference point}
\author{Tomáš Roubal\thanks{Institute of Information Theory and Automation, Czech Academy of Sciences, Prague, Czech Republic, 	 roubal@utia.cas.cz, ORCID ID: 0000-0002-6137-1046}
	\and 
	    Jan Valdman\thanks{Institute of Information Theory and Automation, Czech Academy of Sciences, Prague, Czech Republic,	 jan.valdman@utia.cas.cz, ORCID ID: 0000-0002-6081-5362}
	    \thanks{Department of Mathematics, Faculty of Science, University of South Bohemia, Brani\v{s}ovsk\'a 31, \v{C}esk\'e Bud\v{e}jovice, Czech Republic}
	    }
\date{}
\maketitle

\begin{abstract}
	This paper investigates a path-following method inspired by the semismooth$^*$ approach for solving algebraic inclusions, with a primary emphasis on the role of uniform subregularity. Uniform subregularity is crucial for ensuring the robustness and stability of path-following methods, as it provides a framework to uniformly control the distance between the input and the solution set across a continuous path. We explore the problem of finding a mapping $ x: \mathbb{R} \longrightarrow \mathbb{R}^n $ that satisfies $ 0 \in F(t, x(t)) $ for each $ t \in [0, T] $, where $ F $ is a set-valued mapping from $ \mathbb{R} \times \mathbb{R}^n $ to $ \mathbb{R}^n $.
	
	The paper discusses two approaches: the first considers mappings with uniform semismooth$^*$ properties along continuous paths, leading to a consistent grid error throughout the interval, while the second examines mappings exhibiting pointwise semismooth$^*$ properties at individual points along the path. The uniform strong subregularity framework is integrated into these approaches to strengthen the stability of solution trajectories and improve algorithmic convergence.

\end{abstract}


\section{Introduction}

In 2021, H. Gfrerer and J. Outrata introduced the semismooth$^*$ method in \cite{GO2021}, inspired by the Newton method, for solving problems of finding $x \in \mathbb{R}^n$ such that
\begin{eqnarray}
	\label{eqInclusion}
	0 \in F(x),
\end{eqnarray}
where $F$ is a set-valued mapping between $\mathbb{R}^n$ and $\mathbb{R}^n$ with a closed graph. This method represents an innovation as it is based on the linearization of the set-valued mapping in inclusion \eqref{eqInclusion}, or both the single-valued and set-valued parts in the case of generalized equations (GEs) using limiting coderivative. This contrasts with the Newton-Josephy method \cite{Josephy1979}, where only the single-valued part is linearized in GEs. For their algorithms they need the following:
points $ (x,y) \in \gph\,F $,  sets $\mathcal{A}_{\text{reg}}\, F({x},{y}) \subset \mathbb{R}^{n\times n}\times \mathbb{R}^{n\times n}$ and $\mathcal{G}_{F, \bar{x}}^{\eta, \kappa}(x)\subset \mathbb{R}^n\times \mathbb{R}^n\times\mathbb{R}^{n\times n}\times\mathbb{R}^{n\times n}$ such that
$$
\mathcal{A}_{\text{reg}}\, F({x},{y}):=\left\lbrace (A,B):  ((B)_i^T,(A)_i^T)\in \gph\,D^* F({x},{y}),  i \in \lbrace 1,2,\dots,n  \rbrace \text{ and } A^{-1}\text{ exists}\right\rbrace
$$	
and
\begin{eqnarray*}
	\mathcal{G}_{F, \bar{x}}^{\eta, \kappa}(x):=\lbrace(\hat{x}, \hat{y}, A,B): \Vert(\hat{x}-\bar{x},\hat{y})\Vert \leq \eta \Vert x-\bar{x} \Vert, (A, B)\in\mathcal{A}_{\text{reg}}\, F(\hat{x},\hat{y}), \Vert A^{-1}\Vert \Vert (A\mid B)\Vert_F \leq \kappa \rbrace,
\end{eqnarray*}
where $ \eta, \kappa $ are positive numbers and $ \bar{x}$ is a (exact) solution of \eqref{eqInclusion}. The algorithm has the following form:

\begin{algorithm}
	\label{algGFOUT}
	\begin{enumerate}
		\item Choose a starting point $ x_{0} $; set the iteration counter $ k := 0 $.
		\item If $ 0 \in F(x_{k}) $, stop the algorithm.
		\item Compute $ (\hat{x}_{k}, \hat{y}_{k}) \in \text{gph } F $ close to $ (x_{k}, 0) $ such that $ \mathcal{A}_{\text{reg}}\, F(\hat{x}_{k}, \hat{y}_{k}) \neq \emptyset $.
		\item Select $ (A, B) \in \mathcal{A}_{\text{reg}}\, F(\hat{x}_{k}, \hat{y}_{k}) $ and compute the new iterate $ x_{k+1} = \hat{x}_{k} - A^{-1} B \hat{y}_{k} $.
		\item Set $ k := k + 1 $ and go to 2.
	\end{enumerate}
\end{algorithm}

Under semismoothness$^*$ of the mapping $ F $ at the solution $ \bar{x} $ of $\eqref{eqInclusion}$ and nonemptyness of the set 	$\mathcal{G}_{F, \bar{x}}^{\eta, \kappa}(x)$ for each $x$ near the solution, they proved a convergence theorem.

\begin{theorem}
	Assume that $ F $ is semismooth* at $ (\bar{x}, 0) \in \text{gph } F $ and assume that there are $ \ell, \kappa > 0 $ such that for every $ x \notin F^{-1}(0) $ sufficiently close to $ \bar{x} $ we have $ \mathcal{G}^{\ell,\kappa}_{F,\bar{x}}(x) \neq \emptyset $. Then there exists some $ \delta > 0 $ such that for every starting point $ x_{0} \in \ball_{\mathbb{R}^n}(\bar{x},\delta ) $ Algorithm \ref{algGFOUT} either stops after finitely many iterations at a solution or produces a sequence $ x_{k} $ which converges superlinearly to $ \bar{x} $, provided we choose in every iteration $ (\hat{x}_{k}, \hat{y}_{k}, A, B) \in \mathcal{G}^{\ell,\kappa}_{F,\bar{x}}(x_{k}) $.
\end{theorem}
Let us note that $( x_k)$  converges superlinearly to $\bx$, means that 
\begin{eqnarray*}
	\lim_{k\to \infty} \frac{\Vert x_{k+1}-\bx\Vert}{\Vert x_k-\bx\Vert}=0.
\end{eqnarray*}
Their work continued into several papers \cite{GMOV2023, GO2022, GO2023, GOV2022}.

The semismooth$^*$ method builds on the concept of semismoothness$^*$, which ensures local superlinear convergence under appropriate regularity conditions. This method generalizes the traditional Newton method by incorporating the set-valued nature of the problem, allowing it to handle more complex scenarios typically encountered in optimization and equilibrium models.

A. Dontchev, M. Krastanov, R. Rockafellar, and V. Veliov studied in \cite{DKRV2013} parametric generalized equations (GEs) in the form
\begin{eqnarray}
	\label{eqVarInequality}
	0 \in f(t,x(t)) + N_{K}(x(t)) \quad \text{for each} \quad t \in [0, T],
\end{eqnarray}
where $T > 0$, $f$ is a single-valued mapping between $\mathbb{R} \times \mathbb{R}^n$ and $\mathbb{R}^m$, $K$ is a closed convex subset of $\mathbb{R}^n$, and $N_K$ is the normal cone mapping to the set $K$. They introduced the two-step path-following method for \eqref{eqVarInequality} in the form
\begin{eqnarray*}
	\begin{cases}
		f(t_k, x_k) + h \nabla_t f(t_k, x_k) + \nabla_x f(t_k, x_k) (u_{k+1} - x_k) + N_K(u_{k+1}) \ni 0, \\
		f(t_{k+1}, u_{k+1}) + \nabla_t f(t_{k+1}, u_{k+1}) (x_{k+1} - u_{k+1}) + N_K(x_{k+1}) \ni 0,
	\end{cases}
\end{eqnarray*}
where $h$ is a   discretization step and $x_0$ equals $x(0)$.
This method is designed to track the solution trajectories of \eqref{eqVarInequality} effectively. This method is particularly useful for solving problems, where the solution mapping is set-valued and the system exhibits strong regularity. The method is a combination of Euler prediction and Newton correction steps, providing high accuracy in tracking solution trajectories over the interval.

The method was extended by R. Cibulka and the first author in \cite{CR2018}, for general set-valued mappings $F$ instead of $N_K$ in \eqref{eqVarInequality} and under a weaker regularity assumption. Finally, the path-following method was extended for the problem called Differential generalized equation in \cite{CDKV2018} and \cite{CPR2019} as a combination of the Euler method and one-step the path-following method.

We introduce two one-step path-following methods based on the semismooth$^*$ approach for the problem of finding $ x: \mathbb{R} \longrightarrow \mathbb{R}^n $ such that
\begin{eqnarray}
	 	\label{eqInclusionParametric}
0 \in F(t, x(t)) \quad \text{for each} \quad t \in [0, T],
\end{eqnarray}
where $ T > 0 $ and $ F $ is a set-valued mapping between $ \mathbb{R} \times \mathbb{R}^n $ and $ \mathbb{R}^n $. In the first case, we focus on mappings $ F $ whose graphs have the uniform semismoothness$^*$ property along continuous paths. This results in a uniform grid error over the entire interval. In the second case, we focus on mappings that have the (non-uniform) semismooth$^*$ property at each point of a continuous path.

The paper is organized as follows. Section 2 provides a comprehensive background in variational analysis, covering the essential concepts required for understanding the subsequent sections.

Section 3 introduces the uniform properties for sets and mappings in the spirit of the semismooth$^*$ property, characterizing them in terms of standard (regular and limiting) coderivatives, and thoroughly investigates their relationship to semismooth$^*$ sets from previous studies. Roughly speaking, by the word "uniform" we mean that the constants as well as the size of the neighborhoods, appearing in the corresponding definitions, remain the same for a certain set of mappings and/or points. This section also presents some basic classes of the properties, providing a foundation for the main results collected in the following sections.

In Section 4, we study strong metric subregularity at/around the reference point of set-valued mappings. In particular, we focus on sufficient conditions for uniform strong metric subregularity at/around the reference point. Conditions ensuring uniform strong subregularity along continuous paths are obtained. 

Section 5 introduces the path-following method for the problem \eqref{eqInclusionParametric} under uniform semismooth$^*$ assumptions. The sufficient conditions for a grid error of order $ \mathcal{O}(h) $ are obtained. This section details the algorithmic steps and the theoretical foundation for ensuring accuracy and convergence of the method.

In Section 6, we  introduce the one-step path-following method for the problem \eqref{eqInclusionParametric} under (non-uniform) semismooth$^*$ assumptions. Under weaker conditions, we obtained  a grid error of order $ \mathcal{O}(h) $  for one step. This section emphasizes the adaptability of the method to handle variations in the semismooth$^*$ properties at different points along the path.

Section 7 focuses on applying path-following methods to the electric circuit problems. In this section, we will also check the theoretical results. By using these methods on the circuit problems, we aim to  show how our theoretical findings can be practically applied. Our implementation is available from the MATLAB central server at 
\url{https://www.mathworks.com/matlabcentral/fileexchange/174255} .

 \section{Preliminaries}
We denote the metric spaces by $X$, $Y$, $P$ and the metric in them by $d$. Associated with these metrics, the closed ball and open ball of radius $\delta$ centered at a point $x \in X$ are defined respectively as
 $$
\ball_X[x, \delta] := \{u \in X : d(x, u) \leq \delta\}
\quad \text{and}\quad
 \ball_X(x, \delta) = \{u \in X : d(x, u) < \delta\}.
 $$

 The distance from a point $x \in X$ to a set ${A}$ is denoted by $\dist(x, {A})$ and is defined as the shortest distance between $x$ and any point in ${A}$, expressed as $\dist(x, {A}) = \inf_{u \in {A}}d(u, x).$
 
 Banach spaces, which are a special type of metric space, where the metric is derived from a norm, are denoted by $(X, \|\cdot\|)$ and $(Y, \|\cdot\|)$.  The unit sphere and the closed unit ball in $X$, both centered at the origin, are denoted by $\mathbb{S}_X$ and $\mathbb{B}_X$, respectively.
 
 The graph of a set-valued mapping $F$, represented as $\gph\, F$, comprises all pairs $(x, y)$ such that $y \in F(x)$. Additionally, the domain of $F$, denoted by $\dom\, F$, includes all points $x$ for which the set $F(x)$ is nonempty, indicating the extent of the definition of $F$. The inverse of a set-valued mapping $F$, denoted by $F^{-1}$, is defined such that $y \in F(x)$ implies $x \in F^{-1}(y)$. This is expressed as $ F^{-1}(y) = \{x \in X \mid y \in F(x)\}.$
 
 The $n$-dimensional Euclidean space is denoted by $\mathbb{R}^n$. The norm of $x \in \mathbb{R}^n$ is referred to as the Euclidean norm. The norm in the Cartesian product of Euclidean spaces is defined by $$
 \|(x, y)\| := \max\{\|x\|, \|y\|\}
 \quad
 \text{and}
\quad
 \|(x, u, y)\| := \max\{\|x\|, \|u\|, \|y\|\}
 $$
 for $(x, u, y) \in \mathbb{R}^n \times \mathbb{R}^m \times \mathbb{R}^p$. The dot product in $\mathbb{R}^n$ is denoted by $\langle \cdot, \cdot \rangle$.
 
Matrices are represented by bold uppercase letters such as ${A}$ and ${B}$. The transpose of a matrix ${A}$ is represented by ${A}^T$, which reflects the matrix over its diagonal, turning rows into columns and vice versa. By $e_1,e_2,\dots, e_n$ we denote elements of canonical basis of $\mathbb{R}^n$.

If a square matrix ${A}$ is non-singular, meaning it has an inverse, this inverse is denoted by ${A}^{-1}$. For clarity and specificity in certain contexts, the $i$-th row of matrix $A$ can be denoted as $(A)_i$, emphasizing the row-wise examination of matrix structures.  The notation $(A \mid B)$ denotes the horizontal concatenation of matrices $A$ and $B$ and $I_n$ represents the identity matrix of size $n \times n$.
 
We denote the Jacobian of a mapping $ f \colon \mathbb{R}^n \longrightarrow \mathbb{R}^n $ by $ \nabla f(x) $, and the derivative of a mapping $ p \colon \mathbb{R} \longrightarrow \mathbb{R}^n $ by $ \dot{p} $. The partial Jacobian with respect to $ x $ for a mapping $ f \colon \mathbb{R} \times \mathbb{R}^n \longrightarrow \mathbb{R}^n $ is denoted by $ \nabla_x f(t, x) $.

Further, we  utilize the fundamental concepts of modern variational analysis.

\begin{definition}Let $A$ be a set in $\mathbb{R}^s$ and let $\bar{x} \in A$.
\begin{itemize}
	\item[(i)] The \textit{tangent (contingent, Bouligand) cone} to $A$ at $\bar{x}$ is given by
	$$
	T_A(\bar{x}) := \mathop{\mathrm{Lim\,sup}}_{t \downarrow 0} \frac{A - \bar{x}}{t}.	$$
	\item[(ii)] The set
	$$
	\hat{N}_A(\bar{x}) := (T_A(\bar{x}))^\circ
	$$
	is the \textit{regular (Fréchet) normal cone} to $A$ at $\bar{x}$, and
	$$
	N_A(\bar{x}) := \mathop{\mathrm{Lim\,sup}}_{x \xrightarrow{A} \bar{x}} \hat{N}_A(x)
	$$
	is the \textit{limiting (Mordukhovich) normal cone} to $A$ at $\bar{x}$.
\end{itemize}

In this definition, the term ``$\limsup$'' refers to the Painlevé-Kuratowski \textit{outer (upper) set limit} and $K^\circ$ refers dual cone to $K$.
\end{definition} 
In \cite{GO2021}, new equivalent definition of semismoothness of a set at the reference point was introduced, referred to as semismoothness$^*$.

\begin{definition}Consider a set $A\subset\mathbb{R}^n$ and a point $\bx\in A$. We say that $A$ is
\emph{semismooth}$^*$ at $\bx$ if for each $\varepsilon>0$ there is $\delta>0$ such that
\begin{eqnarray*}
	\vert \left< x^*, x-\bx\right>\vert\leq \varepsilon \Vert x^*\Vert \Vert x-\bx\Vert\quad\text{for each}\quad x\in \ball_{\mathbb{R}^n}[\bx,\delta]\cap A\quad\text{and}\quad x^*\in N_A(x).
\end{eqnarray*}
\end{definition}

 The cones listed above allow us to describe the local behaviour of set-valued maps through various generalized derivatives. 
\begin{definition}
	 Consider a set-valued mapping $F : \mathbb{R}^n \rightrightarrows \mathbb{R}^m$ and let $(\bar{x}, \bar{y}) \in \mathrm{gph}\, F$.
	\begin{itemize}
		\item[\rm (i)] The set-valued mapping $DF(\bar{x}, \bar{y}) : \mathbb{R}^n \rightrightarrows \mathbb{R}^m$ given by $\mathrm{gph}\, DF(\bar{x}, \bar{y}) = T_{\mathrm{gph}\, F}(\bar{x}, \bar{y})$ is called the \textit{graphical derivative} of $F$ at $(\bar{x}, \bar{y})$.
	
		\item[\rm (ii)] The set-valued mapping $\hat{D}^*F(\bar{x}, \bar{y}) : \mathbb{R}^m \rightrightarrows \mathbb{R}^n$ defined by
		$$
		\mathrm{gph}\, \hat{D}^*F(\bar{x}, \bar{y}) = \{(y^*, x^*) \mid (x^*, -y^*) \in \hat{N}_{\mathrm{gph}\, F}(\bar{x}, \bar{y})\}
		$$
		is called the \textit{regular (Fréchet) coderivative} of $F$ at $(\bar{x}, \bar{y})$.
		\item[\rm (iii)] The set-valued mapping $D^*F(\bar{x}, \bar{y}) : \mathbb{R}^m \rightrightarrows \mathbb{R}^n$ defined by
		$$
		\mathrm{gph}\, D^*F(\bar{x}, \bar{y}) = \{(y^*, x^*) \mid (x^*, -y^*) \in N_{\mathrm{gph}\, F}(\bar{x}, \bar{y})\}
		$$
		is called the \textit{limiting (Mordukhovich) coderivative} of $F$ at $(\bar{x}, \bar{y})$.
	\end{itemize}
\end{definition}
Using limiting coderivative, we can define semismoothness$^*$ for set-valued mapping.
\begin{definition} Consider a set-valued mapping $F: \mathbb{R}^n\tto \mathbb{R}^m$ and a point $(\bx,\by)\in \gph\,F$. We say that $F$ is
	\emph{semismooth}$^*$ at $(\bx,\by)$ if for each $\varepsilon>0$ there is $\delta>0$ such that for each $(x,y)\in \big(\ball_{\mathbb{R}^n}[\bx,\delta]\times \ball_{\mathbb{R}^n}[\by,\delta]\big)\cap \gph\,F$ we have
	\begin{eqnarray*}
		\vert \left< x^*, x-\bx\right>-\left< y^*, y-\by\right>\vert\leq \varepsilon \Vert (x^*, y^*)\Vert \Vert (x,y)-(\bx,\by)\Vert\quad\text{for each}\quad (y^*,x^*)\in   \gph\,D^* F(x,y).
	\end{eqnarray*}
\end{definition}

In modern variational analysis, examining the regularity of set-valued mappings is essential for interpreting various mathematical models, especially in fields such as optimization, control theory, and economics. The regularity of these mappings refers to the characteristics that determine the local behaviour of the mapping around a point in its domain. Here, we focus solely on properties that are relevant to our research.
\begin{definition}\label{defRegularity}Let $X$ and $Y$ be metric spaces. 
		Let $F : X \rightrightarrows Y$ be a set-valued mapping and let $(\bar{x}, \bar{y}) \in \mathrm{gph}\, F$
		 be a given point. We say that $F$ is:
		\begin{enumerate}
			\item[\rm (i)]  \textit{metrically subregular} at $(\bar{x}, \bar{y})$ if there exists $\kappa \geq 0$ along with some \nn $U$ of $\bar{x}$ such that
			$$
			\mathrm{dist}(x, F^{-1}(\bar{y})) \leq \kappa\, \mathrm{dist}(\bar{y}, F(x)) \quad  \text{for each}\quad x \in U;
			$$
			\item[\rm (ii)] \textit{strongly metrically subregular} at $(\bar{x}, \bar{y})$ if it is metrically subregular at $(\bar{x}, \bar{y})$ and there exists a \nn $U$ of $\bar{x}$ such that $F^{-1}(\bar{y}) \cap U = \{\bar{x}\}$;
			\item[\rm (iii)] (strongly) \textit{metrically subregular around} $(\bar{x}, \bar{y}) \in \mathrm{gph}\, F$ if there is a \nn $W$ of $(\bar{x}, \bar{y})$ such that $F$ is (strongly) metrically subregular at every point $(x, y) \in \mathrm{gph}\, F \cap W$;
			
			\item[\rm (iv)] \textit{metrically regular} around $(\bar{x}, \bar{y})$ if there is $\kappa \geq 0$ together with \nns $U$ of $\bar{x}$ and $V$ of $\bar{y}$ such that
			$$
			\mathrm{dist}(x, F^{-1}(y)) \leq \kappa\, \mathrm{dist}(y, F(x)) \quad \text{for each}\quad (x, y) \in U \times V;
			$$
			\item[\rm (v)]  \textit{strongly metrically regular} around $(\bar{x}, \bar{y})$ if it is metrically regular around $(\bar{x}, \bar{y})$ and $F^{-1}$ has a single-valued localization around $(\bar{y}, \bar{x})$, i.e., there are open \nns $V$ of $\bar{y}$, $U$ of $\bar{x}$ and a mapping $h : V \longrightarrow \mathbb{R}^n$ with $h(\bar{y}) = \bar{x}$ such that $\mathrm{gph}\, F \cap (U \times V) = \{(h(y), y) \mid y \in V\}$.
		\end{enumerate}
\end{definition}
	Strong metric subregularity around the reference point was first introduced in \cite{GO2022}; also see \cite{DR2014, Ioffe2017} for the additional properties.

\section{Uniform semismoothness$^*$}

In this section, we present an extension of the properties for sets in the spirit of semismoothness$^*$.

Subsmoothness at the reference point and uniform subsmoothness for sets were first introduced in \cite{ADT2005}. Subsmoothness refers to a set being smooth in a localized sense at a specific point, while uniform subsmoothness extends this property to hold uniformly over a set of points.

We present a restricted version of uniform subsmoothness for sets, focusing on specific conditions under which uniform subsmoothness can be guaranteed. 

Moreover, we introduce the uniform version of semismooth$^*$ property for sets, which was initially introduced in \cite{GO2021}. This uniform version extends the semismooth$^*$ property to hold uniformly over a set, similar to the concept of uniform subsmoothness for sets.

\begin{definition}\label{defSemiSubA} Consider  sets $A\subset\mathbb{R}^n$, $U\subset A$, and a point $\bx\in A$. We say that $A$ is:
	\begin{enumerate}
		\item[\rm (i)]  \emph{subsmooth} at $\bx$ if for each $\varepsilon>0$ there is $\delta>0$ such that
		\begin{eqnarray*}
			\left< x^*, x-u\right>\leq \varepsilon \Vert x^*\Vert \Vert x-u\Vert\quad\text{for each}\quad x, u\in \ball_{\mathbb{R}^n}[\bx,\delta]\cap A\quad\text{and}\quad x^*\in N_A(x);
		\end{eqnarray*}
		
		\item[\rm (ii)]  \emph{uniformly subsmooth} on $U$ if for each $\varepsilon>0$ there is $\delta>0$ such that for each $v\in U$ we have
		\begin{eqnarray*}
			\left< x^*, x-u\right>\leq \varepsilon \Vert x^*\Vert \Vert x-u\Vert\quad\text{for each}\quad x, u\in \ball_{\mathbb{R}^n}[v,\delta]\cap A\quad\text{and}\quad x^*\in N_A(x);
		\end{eqnarray*}
		\item[\rm (iii)]  \emph{uniformly semismooth}$^*$ on $U$ if for each $\varepsilon>0$ there is $\delta>0$ such that for each $u\in U$ we have
		\begin{eqnarray*}
			\vert \left< x^*, x-u\right>\vert\leq \varepsilon \Vert x^*\Vert \Vert x-u\Vert\quad\text{for each}\quad x\in \ball_{\mathbb{R}^n}[u,\delta]\cap A\quad\text{and}\quad x^*\in N_A(x).
		\end{eqnarray*}
	\end{enumerate}
\end{definition}
	
A. Jourani in \cite[Theorem 2]{Jourani2007} showed that every subanalytical  set is  semismooth$^*$ at  each point of $x\in A$. 
A subset $ A $ of $ \mathbb{R}^n $ is called subanalytic if for each point of $ A $, there is a \nn $U$ such that $ A \cap U $ can be represented as the projection of a relatively compact semianalytic set. In other words, there is a real analytic manifold $ N $ and a relatively compact semianalytic subset $ A $ of $ \mathbb{R}^n \times N $ such that $ X \cap U = \pi(A) $, where $ \pi: \mathbb{R}^n \times N \longrightarrow \mathbb{R}^n $ is the projection map. A subset $ A $ of $ \mathbb{R}^n $ is defined as semianalytic if for every point $ a \in \mathbb{R}^n $, there exists a \nn $ V $ such that 
$$
 A\cap V = \bigcup_{i=1}^p \bigcap_{j=1}^q \{ x \in V \mid f_{ij}(x) \ \sigma_{ij} \ 0 \},
$$
where each $ f_{ij} $ is a real analytic function on $ V $ and each $ \sigma_{ij} $ is either $= $ or $ > $.

\begin{proposition} Consider a closed set $A\subset \mathbb{R}^n$. If the set $A$ is subanalytical, then for each $x\in A$ the set $A$ is semismooth$^*$ at $x$.
\end{proposition}
The following lemma comes from the previous proposition.
\begin{lemma}
	 Consider a set $A\subset \mathbb{R}^n$.  If the $A$ is closed and convex, then is subanalytical. In particular, for each $x\in A$ the set  $A$ is semismooth$^*$ at $x$.
\end{lemma}

If finitely many closed sets are semismooth$^*$ at the same reference point, then their union is also semismooth$^*$ at the reference point, see \cite[Proposition 3.5]{GO2021}.
\begin{lemma}\label{lemUnionSemismooth}
	Consider closed sets $A_i \subset \mathbb{R}^n$ for $i = 1, 2, 3, \dots, N$ and a point $\bx \in A := \cup_{i=1}^N A_i$. If each set $A_i$, with $\bx \in A_i$, is semismooth$^*$ at $\bx$, then $A$ is semismooth$^*$ at $\bx$, where $A_i$ is semismooth$^*$ at $\bx$.
\end{lemma}

This lemma gives us that every finite union of closed subanalytical sets is semismooth$^*$ at every point. In particular, every  closed polyhedral set has the same property.

The following remark comes from \cite{GO2021}.
\begin{remark}
	Let us note that, by Lemma \ref{lemUnionSemismooth}, the union of finitely many closed convex sets is semismooth* at every point. We get:
	\begin{enumerate}
		\item a set-valued mapping $ F: \mathbb{R}^n \rightrightarrows \mathbb{R}^m $, with a closed convex graph, is semismooth* at every point $(\bar{x}, \bar{y}) \in \operatorname{gph} F $;
		\item a set-valued mapping $ F: \mathbb{R}^n \rightrightarrows \mathbb{R}^m $, with a polyhedral graph, is semismooth* at every point $(\bar{x}, \bar{y}) \in \operatorname{gph} F $. Specifically, for every closed convex polyhedral set $ D \subseteq \mathbb{R}^s $, the normal cone mapping $ N_D $ is semismooth* at every point of its graph.
	\end{enumerate}
\end{remark}

Limiting coderivative of a single-valued mapping can be described by its gradient for smooth case or by Bouligand subdifferential in Lipschitz case, see \cite[8.34 Example]{RW1998}.
\begin{remark} \label{remCodevSingle}
If a single-valued mapping $f:\mathbb{R}^n \to \mathbb{R}^n$ is continuously differentiable at ${x}$, then
$$
\gph\,D^*f({x}, f({x})) = \left\{ (v, \nabla f({x})^T v) : v \in \mathbb{R}^n \right\}.
$$
If $f$ is Lipschitz continuous around ${x}$, then
$$
\left\{ (v, A^T v) : v \in \mathbb{R}^n \text{ and } A \in \partial_B f({x}) \right\} \subset \gph\,D^*f({x}, f({x})),
$$
where $\partial_B f({x}) := \left\{ A \in \mathbb{R}^{n \times n} : \nabla f(x_k) \to A, x_k \to {x}, \text{ and } \{x_k\} \subset \Omega_f \right\}$, and the set $\Omega_f\subset\mathbb{R}^n$ contains points, where $f$ is Fréchet differentiable.
\end{remark}

\begin{lemma}
	\label{lemContinousDifSemismooth}
	Consider a single-valued mapping $ f: \mathbb{R}^n \longrightarrow \mathbb{R}^n $ which is continuously differentiable  around  $ \bx\in \mathbb{R}^n$ and a constant $\varepsilon>0$.  If there is  $ \delta > 0 $ such that for each  $x \in \ball_{\mathbb{R}^n}[\bx,\delta]$ we have
	\begin{eqnarray*}
		\| f(x) - f(\bx) - \nabla f (x) (x - \bx) \| \leq \varepsilon \| x - \bx \|,
	\end{eqnarray*}
	then for each $x\in \ball_{\mathbb{R}^n}[\bx,\delta]$ and  each $(y^*,x^*)\in \gph\,D^* f(x,f(x))$ we have
	\begin{eqnarray*}
		\vert	\langle x^*, x - \bx \rangle - \langle y^*, f(x) - f(\bx) \rangle\vert\leq \varepsilon  \|( x^*, y^* )\| \|(x, f(x)) - (\bx, f(\bx))\| .
	\end{eqnarray*}
\end{lemma}
\begin{proof} Find $\delta>0$. Fix any  $x \in \ball_{\mathbb{R}^n}[\bx,\delta]$ and   any $(y^*, x^*) \in \gph\, D^*f(x,f(x))$. Then, by Remark \ref{remCodevSingle}, we get $x^* =\nabla f(x)^T y^*$ and we obtain
	\begin{eqnarray*}
		&&	\langle x^*, x - \bx \rangle - \langle y^*, f(x) - f(\bx) \rangle = \langle y^*, \nabla f(x)(x - \bx) - (f(x) - f(\bx)) \rangle \\
		&\leq& \|y^*\| \|\nabla f(x) (x - \bx) - (f(x) - f(\bx))\| \leq\varepsilon \|y^*\| \|x - \bx\| \leq  \varepsilon \|( x^*, y^* )\|\|(x, f(x)) - (\bx, f(\bx))\|.
	\end{eqnarray*}
\end{proof}

If $\gph\,F$ of a set-valued mapping $F:\mathbb{R}^n\tto \mathbb{R}^m$ has a property from Definition \ref{defSemiSubA}, then we say that the mapping $F$ has the same property.
\begin{definition} Consider  a set-valued mapping $F:\mathbb{R}^n\tto \mathbb{R}^m$, a set  $W\subset\gph\,F$, and a point $(\bx, y)\in \gph\,F$. We say that $F$ is:
	\begin{enumerate}
		\item[\rm (i)] \emph{subsmooth} at $(\bx,\by)$ if for each $\varepsilon>0$ there is $\delta>0$ such that for each $(x,y), (u,z)\in \big(\ball_{\mathbb{R}^n}[\bx,\delta]\times \ball_{\mathbb{R}^n}[\by,\delta]\big)\cap \gph\,F$ we have
		\begin{eqnarray*}
			\left< x^*, x-u\right>-\left< y^*, y-z\right>\leq \varepsilon \Vert (x^*, y^*)\Vert \Vert (x,y)-(u,z)\Vert \quad\text{for each}\quad (y^*,x^*)\in   \gph\,D^* F(x,y);
		\end{eqnarray*}

		\item[\rm (ii)]  \emph{uniformly subsmooth} on $W$ if for each $\varepsilon>0$ there is $\delta>0$ such that for each $(v, w)\in W$ and  each $(x,y), (u,z)\in \big(\ball_{\mathbb{R}^n}[v,\delta]\times \ball_{\mathbb{R}^n}[w,\delta]\big)\cap \gph\,F$ we have
		\begin{eqnarray*}
			\left< x^*, x-u\right>-\left< y^*, y-z\right>\leq \varepsilon \Vert (x^*, y^*)\Vert \Vert (x,y)-(u,z)\Vert\quad\text{for each}\quad (y^*,x^*)\in  \gph\,D^* F(x,y);
		\end{eqnarray*}
				\item[\rm (iii)] \emph{uniformly semismooth}$^*$ on $W$ if for each $\varepsilon>0$ there is $\delta>0$ such that for each $(u, z)\in W$ and each $(x,y)\in \big(\ball_{\mathbb{R}^n}[u,\delta]\times \ball_{\mathbb{R}^n}[z,\delta]\big)\cap \gph\,F$ we have
		\begin{eqnarray*}
			\vert \left< x^*, x-z=u\right>-\left< y^*, y-z\right>\vert\leq \varepsilon \Vert (x^*, y^*)\Vert \Vert (x,y)-(u,z)\Vert\quad\text{for each}\quad (y^*,x^*)\in   \gph\,D^* F(x,y).
		\end{eqnarray*}
	\end{enumerate}
\end{definition}

\begin{remark} \label{remInv}
	Note that for a set-valued mapping $F:\mathbb{R}^n \tto \mathbb{R}^n$ we have $(y^*, x^*) \in D^* F(x, y)$ if and only if $(-x^*, -y^*) \in D^* F^{-1}(y, x)$.
\end{remark}

If the inverse of a set-valued mapping $F$ has a continuously differentiable localization, then $F$ possesses the uniform semismoothness$^*$ property along this localization.
\begin{proposition}
	\label{proInverseSmooth}
	Consider a set-valued mapping $F:\mathbb{R}^n \tto\mathbb{R}^n$,  a compact set $K\subset\mathbb{R}^n$, a number $\alpha>0,$ and a single-valued mapping $f : (K+\alpha \ball_{\mathbb{R}^n}) \longrightarrow \mathbb{R}^n$ such that $\mathrm{gph}\, F \cap (U \times (K+\alpha\ball_{\mathbb{R}^n})) = \{(f(y), y) : y \in K+\alpha \ball_{\mathbb{R}^n}\}$ for  $U:=f^{-1}(K+\alpha\ball_{\mathbb{R}^n})$. Moreover, assume that $f$ is continuously differentiable on $K$. Then $F$ is uniformly semismooth$^*$ on $U\times K$.
\end{proposition}
\begin{proof}
Since $f$ is continuously differentiable and $\nabla f$ is continuous on $K$, for each   $\varepsilon>0$ and each  $\by \in K$ there is $\delta\in (0,\alpha)$  we have that for each $y,v \in\ball_{\mathbb{R}^n}[\by,\delta]$ we have
\begin{eqnarray*}
\Vert f(y)-f(v)-\nabla f(\by)(y-v)\Vert \leq \varepsilon/2\Vert y-v\Vert \quad\text{and}\quad
	\Vert \nabla f(y)-\nabla f(\by)\Vert\leq \varepsilon/2.
\end{eqnarray*}
Fix any $\varepsilon>0$ and any $\by\in K$.
Then for each $y, v\in \ball_{\mathbb{R}^n}[\by,\delta]$ we get
\begin{eqnarray*}
&&	\Vert f(y)-f(v)-\nabla f(y)(y-v)\Vert\leq 	\Vert f(y)-f(v)-\nabla f(\by)(y-\by)\Vert\\ &+&\Vert \nabla f(y)-\nabla f(\by)\Vert \Vert y-v\Vert	\leq \varepsilon\Vert y-v\Vert.
\end{eqnarray*}
To sum up, we showed that for each $\by\in K$ there is ${\delta}\in(0,\alpha)$ such that for each $y,v \in \ball_{\mathbb{R}^n}[\by,{\delta}]$ we get
\begin{eqnarray*}
	\Vert f(y)-f(v)-\nabla f(y)(y-v)\Vert\leq  \varepsilon\Vert y-v\Vert.
\end{eqnarray*}
Then the system of sets $\lbrace \ball_{\mathbb{R}^n}(\by,{\delta}_{\by}/2): \quad\by\in K\rbrace$ is open covering of $K$. By compactness of $K$, there are $\by_1,\by_2,\dots, \by_N\in K$ such that
\begin{eqnarray*}
	\cup_{i=1}^N \ball_{\mathbb{R}^n}(\by_i,{\delta}_{\by_i}/2) \supset K.
\end{eqnarray*}
Let $\delta:=\min\lbrace {\delta}_{\by_1}, {\delta}_{\by_2}, \dots, {\delta}_{\by_N}  \rbrace/2$. Fix any $\by\in K$ and any $y\in \ball[\by, \delta]$. Then there is $i \in \lbrace 1,2, \dots, N\rbrace$  such that $\by \in \ball_{\mathbb{R}^n}[\by_i,{\delta}_{\by_i}/2]$ and so $y\in \ball_{\mathbb{R}^n}[\by_i, {\delta}_{\by_i}]$ because
\begin{eqnarray*}
	\Vert y-\by_i\Vert \leq \Vert y-\by\Vert +\Vert \by -\by_i\Vert\leq \delta+{\delta}_{\by_i}/2\leq {\delta}_{\by_i}.
\end{eqnarray*}
We showed that for each $\by\in K$ and each $y\in \ball_{\mathbb{R}^n}[\by,\delta]$ we have
\begin{eqnarray*}
	&&	\Vert f(y)-f(\by)-\nabla f(y)(y-\by)\Vert\leq \varepsilon\Vert y-\by\Vert.
\end{eqnarray*}
Fix any $\by \in K$. 
By Lemma \ref{lemContinousDifSemismooth},  we have that for each $y\in \ball_{\mathbb{R}^n}[\by,\delta]$ and  each $(y^*,x^*)\in \gph\,D^* f(y,f(y))$ we have
\begin{eqnarray*}
	\vert	\langle x^*, y - \by \rangle - \langle y^*, f(y) - f(\by) \rangle\vert\leq \varepsilon  \|( x^*, y^* )\| \|( f(y),y) - ( f(\by),\by)\|.
\end{eqnarray*}
Fix any $(x,y) \in\big(\ball_{\mathbb{R}^n}[f(\by),\delta]\times \ball_{\mathbb{R}^n}[\by,\delta]\big)\cap \gph\,F$ and any $(y^*,x^*)\in \gph\,D^* F(x,y)$. Then $x=f(y)$ and, by Remark \ref{remInv}, $(-x^*, -y^*)\in\gph\, D^*f(x,f(x))$.  We conclude that
\begin{eqnarray*}
&&	\vert \left< x^*, x-f(\by)\right>-\left< y^*, y-\by\right>\vert=\vert \left< -x^*, f(y)-f(\by)\right>-\left<- y^*, y-\by\right>\vert\\
&\leq&  \varepsilon  \|( x^*, y^* )\| \|( f(y),y) - ( f(\by),\by)\| =\varepsilon  \|( x^*, y^* )\| \|( x,y) - ( f(\by),\by)\|.
\end{eqnarray*}
\end{proof}
\begin{remark}\label{remCodevInv} Note that if $F, K,$ and $U$ are defined as in the previous proposition, then  for $(x,y)\in U\times K$ we have that $\gph\,D^*F(x, y)=\lbrace ( \nabla f(y)^T u, u):\quad u \in \mathbb{R}^n\rbrace$, when $f$ is continuously differentiable at $y$, and $\lbrace ( A^T u, u):\quad u \in \mathbb{R}^n\text{ and }A\in \partial_B f(y)\rbrace\subset\gph\,D^*F(x, y)$, when $f$ is Lipschitz continuous around $y$.
\end{remark}

A set-valued mapping  with closed convex set is uniformly subsmooth on its graph and the symmetry of limiting coderivate at some subset  of its graph implies that the mapping is uniformly semismooth$^*$ on the subset.
\begin{theorem}
	Consider a set-valued mapping $F:\mathbb{R}^n\tto\mathbb{R}^n$ such that $\gph\,F$ is convex and closed. Then $F$ is uniformly subsmooth on $\gph\,F$. Moreover, if for some $W \subset \gph\,F$ and some $\delta>0$ and each $(x,y)\in (W+\delta\ball_{\mathbb{R}^n})\cap \gph\,F$ we have that $(y^*,x^*)\in   \gph\,D^* F(x,y)$ implies that $(-y^*,-x^*)\in   \gph\,D^* F(x,y)$, then $F$ is uniformly semismooth$^*$ on $W$.
\end{theorem}
\begin{proof}
	Since $\gph F$ is convex, for each $(x, y) \in \gph F$, the coderivative definition yields
	$$
	\gph D^* F(x, y) = \{(y^*, x^*) \in \mathbb{R}^n \times \mathbb{R}^n : \langle x^*, u - x \rangle - \langle y^*, z - y \rangle \leq 0 \text{ for all } (u, z) \in \gph F\}.
	$$
	Then for any $\varepsilon > 0$, any $\delta > 0$, any $(\bx, \by) \in \gph F$,  any $(x, y), (u, z) \in (\ball_{\mathbb{R}^n}[\bx, \delta] \times \ball_{\mathbb{R}^n}[\by, \delta]) \cap \gph F$, and any $(y^*, x^*) \in \gph D^* F(x, y)$ we have
	$$
	\langle x^*, u - x \rangle - \langle y^*, z - y \rangle \leq 0 \leq \varepsilon \| (x^*, y^*) \| \| (u, z) - (x, y) \|.
	$$
	
	Further, assuming $\delta > 0$ and given a set $W$ as stated, fix any $\varepsilon > 0$, any $(\bx, \by) \in W$, any $(x, y) \in (\ball_{\mathbb{R}^n}[\bx, \delta] \times \ball_{\mathbb{R}^n}[\by, \delta]) \cap \gph F$, and any $(y^*, x^*) \in \gph D^* F(x, y)$. By the symmetry of $\gph D^* F(x, y)$, we establish
	$$
	|\langle x^*, x - \bx \rangle - \langle y^*, y - \by \rangle| \leq \varepsilon \| (x^*, y^*) \| \| (x, y) - (\bx, \by) \|,
	$$
	which concludes the proof.
\end{proof}

The uniform semismoothnness$^*$ is stable under continuously differentiable perturbation along continuous path. 

\begin{theorem}
	Let $F: \mathbb{R}^n\tto \mathbb{R}^n$ be a set-valued mapping with a closed graph and  $f:[0,T]\times \mathbb{R}^n\longrightarrow \mathbb{R}^n$ be continuously differentiable mapping for $T>0$. Consider a continuous mapping $x:[0,T]\longrightarrow \mathbb{R}^n$ such that  for each $t\in [0,T]$ we have $0\in f(t,x(t))+F(x(t))$ and a set $W:=\lbrace(x(t),-f(t, x(t))) : t \in [0,T]\rbrace.$ Suppose that $F$ is uniformly semismooth$^*$ on $W$, then for each $\varepsilon>0$ there is $\delta>0$ such that for each $t\in [0,T]$,  each $(x,y)\in\big(\ball_{\mathbb{R}^n}[x(t),\delta]\times \ball_{\mathbb{R}^n}[0,\delta]\big)\cap \gph\,(f(t,\cdot)+F)$, and  each $(y^*, x^*)\in D^* (f(t,\cdot)+F)(x,y)$ we have
	$$
\vert \left< x^*, x-x(t)\right>-\left< y^*, y\right>\vert\leq	\varepsilon\Vert (x^*, y^*)\Vert \Vert (x,y)-(x(t),0)\Vert 
	$$

\end{theorem}
\begin{proof}
	Let $\ell\geq 0$ be a  Lipschitz constant of $f$ on $\cup_{t\in [0,T]} \ball_\mathbb{R}[t,\gamma]\times\ball_{\mathbb{R}^n}[x(t),\gamma]$ for some $\gamma>0$. Fix any $\varepsilon>0$. Since $f$ is continuously differentiable, see proof of Proposition \ref{proInverseSmooth}, there exists $\delta_1\in (0,\gamma)$ such that for each $t\in[0,T]$ and each $x \in \ball_{\mathbb{R}^n}[x(t),\delta_1]$ we have
	\begin{eqnarray*}
		\Vert f(t,x)-f(t,x(t))-\nabla_x f(t,x)(x-x(t))\Vert \leq\tfrac{ \varepsilon}{3}\Vert x-x(t)\Vert.
	\end{eqnarray*}
	Let $c:=\sup_{t\in [0,T], x\in \ball_{\mathbb{R}^n}[x(t),\delta_1]} \Vert \nabla_x f(t,x)^T\Vert$. Since $f$ is continuously differentiable, we have $c<\infty$.
 Fix any  $\widehat{\varepsilon}>0$ such  that $\widehat{\varepsilon}(1+c)\ell\leq \varepsilon/3$ and $\widehat{\varepsilon}(1+c)\leq \varepsilon/3$.  Since $F$ is uniformly semismooth$^*$ on $W$, find $\delta_2>0$  such that for each $t\in [0,T]$, each $(x,y)\in \big(\ball_{\mathbb{R}^n}[x(t),\delta_2]\times \ball_{\mathbb{R}^n}[-f(t,x(t)),\delta_2] \big)\cap \gph\,F$ and each $(y^*,x^*)\in \gph\,D^* F(x,y)$ we have
	\begin{eqnarray*}
		\vert \left< x^*,x-x(t)\right>-\left< y^*,y+f(t,x(t)) \right>\vert\leq \widetilde{\varepsilon} \Vert (x^*, y^*)\Vert \Vert(x,y)-(x(t),-f(t,x(t)))\Vert.  
	\end{eqnarray*}
	
	Let $\delta:=\min \lbrace \delta_1, \delta_2\rbrace/(1+\ell)$. Fix any $t\in [0,T]$, any $(x,y)\in\big(\ball_{\mathbb{R}^n}[x(t),\delta]\times \ball_{\mathbb{R}^n}[0,\delta]\big)\cap \gph\,(f(t,\cdot)+F)$, and any $(y^*, x^*)\in D^* (f(t,\cdot)+F)(x,y)$. 	Since $f$ is continuously differentiable, by \cite[10.43 Exercise]{RW1998}, we have
	\begin{eqnarray*}
		D^* (f(t,\cdot)+F)(x,y)= \nabla_x f(t, x)^T+ D^* F(x,y-f(t,x)).
	\end{eqnarray*}
	Then $y-f(t,x)\in \ball_{\mathbb{R}^n}[-f(t,x(t)),\delta_2]$, since $\Vert y-f(t,x)+f(t,x(t))\Vert\leq \delta+\delta\ell \leq \delta_2$. To sum up, we get
	\begin{eqnarray*}
		&&	\vert \left< x^*,x-x(t)\right>-\left< y^*,y\right>\vert=	\vert \left< x^*-\nabla_x f(t,x)^Ty^*+\nabla_x f(t,x)^Ty^*,x-x(t)\right>-\left< y^*,y\right>\vert\\
		&=& \vert \left< x^*-\nabla_x f(t,x)^Ty^*,x-x(t)\right>-\left< y^*,y-\nabla_x f(t,x)(x-x(t))\right>\vert\\
		&\leq&\vert \left< x^*-\nabla_x f(t,x)^Ty^*,x-x(t)\right>-\left< y^*,y- f(t,x)+f(t,x(t))\right>\vert\\
		&+&\vert\left< y^*,f(t,x)-f(t,
		x(t))-\nabla_x f(t,x)(x-x(t))\right>\vert\\
		&\leq &\widetilde{\varepsilon} \Vert (x^*-\nabla_x f(t,x)^Ty^*, y^*)\Vert \Vert(x,y-f(t,x))-(x(t),-f(t,x(t)))\Vert +\varepsilon/3 \Vert y^*\Vert \Vert x-x(t)\Vert\\
		&\leq &\widetilde{\varepsilon}( \Vert (x^*, y^*)\Vert+c\Vert y^*\Vert ) \Vert(x,y-f(t,x))-(x(t),-f(t,x(t)))\Vert +\varepsilon/3 \Vert y^*\Vert \Vert x-x(t)\Vert\\
		&\leq &\widetilde{\varepsilon}(1+c) \Vert (x^*, y^*)\Vert \left(\Vert(x,y)-(x(t),0)\Vert+\ell\Vert x-x(t)\Vert\right) +\varepsilon/3 \Vert y^*\Vert \Vert x-x(t)\Vert\\
		&\leq &\varepsilon/3 \Vert (x^*, y^*)\Vert \Vert(x,y)-(x(t),0))\Vert+\varepsilon/3 \Vert  y^*\Vert \Vert x-x(t)\Vert +\varepsilon/3 \Vert y^*\Vert \Vert x-x(t)\Vert\\
		&\leq &\varepsilon \Vert (x^*, y^*)\Vert \Vert(x,y)-(x(t),0)\Vert.
	\end{eqnarray*}
\end{proof}

\section{Uniform strong subregularity}

In this section, we are  investigating uniform strong metric subregularity on compact subsets of Banach spaces of mappings which are defined as a sum of a single-valued (possibly nonsmooth) mapping and a set-valued mapping. We are following ideas of the proofs from \cite[Section 2]{CPR2019}.

Further, we present a statement concerning perturbed strong metric subregularity on a set.

\begin{theorem}
	\label{thmStabilitySubregularity}
	Let $(X,\Vert \cdot\Vert)$ and $(Y,\Vert \cdot\Vert)$ be Banach spaces. Consider a set-valued mappings $F:X\tto Y$ and a point $(\bx, \by)\in \gph\,F$. Assume that there are positive constants $a,b,$ and $\kappa$ such that  for each $(x,y)\in \big(\ball_X[\bx, a]\times \ball_Y[\by, b])\cap \gph\,F$ there is $r>0$ such that for each $u\in \ball_X[x, r]$ we have
	\begin{eqnarray}
		\label{eqConstrainedSubreg}
		\Vert u-x\Vert \leq \kappa \dist(y,F(u)\cap \ball_Y[\by, b]).
	\end{eqnarray}
	Let $\mu>0$ be such that $\kappa \mu <1$ and let $\kappa'>\kappa/(1-\kappa \mu)$. Then for every positive $\alpha$ and $\beta$ such that
$		2\alpha\leq a$ and $	2\beta+\mu \alpha \leq b$	and for every mapping $g:X\longrightarrow Y$ satisfying
	\begin{eqnarray*}
	\Vert g(\bx)\Vert\leq\beta\quad\text{and}\quad	\Vert g(x)-g(u)\Vert \leq \mu\Vert x-u\Vert\quad\text{for each}\quad x,u \in \ball_X[\bx,\alpha],
	\end{eqnarray*}
	 we have that for each $y\in \ball_Y[\by,\beta]$ and each $x\in (g+F)^{-1}(y)\cap \ball_X[\bx, \alpha]$  there is $r\in (0,\alpha]$  such that each $u\in\ball_X[x,r]$ and   each   $ v\in  (g+F)(u)\cap \ball_Y[\by, \beta]$   we have
	 \begin{eqnarray*}
	 	\Vert u-x\Vert\leq \kappa' \Vert y-v\Vert.
	 \end{eqnarray*}
	\end{theorem}
	\begin{proof}
		Fix any $\alpha>0$ and $\beta>0$ and any mapping $g$ as in the conclusion. Then fix any $y\in \ball_Y[\by,\beta]$. Fix any $x\in (g+F)^{-1}(y)\cap \ball_X[\bx, \alpha]$ and find a corresponding $r\in (0, \alpha]$ such that \eqref{eqConstrainedSubreg}, for each $u\in \ball_X[x, r]$, holds.
Fix any $u\in \ball_X[x,r]$.
Then $u\in \ball_X[\bx, a]$ since
\begin{eqnarray*}
	\Vert u-\bx\Vert \leq\Vert u-x\Vert+\Vert x-\bx\Vert\leq r+{\alpha}\leq 2\alpha\leq a.
\end{eqnarray*}
Note that $ \ball_Y[\by,\beta]\subset \ball_{Y}[\by+g(x), b]$, since
for each $v \in  \ball_Y[\by,\beta]$ we have
\begin{eqnarray*}
	\Vert\by+ g(x)- v\Vert\leq 	\Vert\by-  v\Vert +\Vert g(x)-g(\bx)\Vert +\Vert g(\bx)\Vert\leq 	\beta+\mu r +\beta \leq 	2\beta+\mu \alpha  \leq b;
\end{eqnarray*}
so is $y-g(x)\in \ball_{Y}[\by, b]$. If $(g(u)+F(u))\cap \ball_Y[\by, \beta] =\emptyset$, we are done. If not, fix any $v\in (g(u)+F(u))\cap \ball_Y[\by, \beta]$. Then
\begin{eqnarray*}
	\Vert u-x\Vert&\leq& \kappa \dist(y-g(x),F(u)\cap \ball_Y[\by, b])\leq \kappa \dist(y-g(u),F(u)\cap \ball_Y[\by, b])+\kappa\Vert g(u)-g(x)\Vert\\
	&\leq &  \kappa \dist(y,(g(u)+F(u))\cap \ball_Y[\by+g(u), b])+ \kappa\mu \Vert u -x\Vert\\&\leq&  \kappa \dist(y,(g(u)+F(u))\cap \ball_Y[\by, \beta])+ \kappa\mu  \Vert u-x\Vert\\
	&\leq&  \kappa \Vert y -v \Vert+ \kappa\mu \Vert u-x\Vert.
\end{eqnarray*}
Taking into account that $\kappa\mu<1$ and $\tfrac{\kappa}{1-\kappa\mu}<\kappa'$, we obtain
\begin{eqnarray*}
	\Vert u-x\Vert\leq \tfrac{\kappa}{1-\kappa\mu} \Vert y -v\Vert\leq  \kappa' \Vert y-v\Vert.
\end{eqnarray*}
	\end{proof}
	
Following proposition using ideas from	\cite{CDK2018} and gives us stability under set-valued perturbation of strong metric subregularity at the reference point.
	\begin{proposition}
		\label{propStability}
		Let $(X,\Vert \cdot\Vert)$ and $(Y,\Vert \cdot\Vert)$ be Banach spaces. Consider a set-valued mappings $F:X\tto Y$ and a point $(\bx, \by)\in \gph\,F$. Assume that there are $\kappa>0$ and  $\alpha>0$ such that $F$ is strongly metrically subregular at $(\bx,\by)$ with the constant $\kappa$ and  the \nn $\ball_{X}[\bx,\alpha]$. Let $\mu>0$ be such that $\kappa \mu <1$ and let $\kappa'>\kappa/(1-\kappa \mu)$.
		
		Then for each $\beta\in (0,\alpha]$	and for each set-valued mapping $G:X\tto Y$ satisfying
		\begin{eqnarray*}
			G(\bx)=\lbrace \bz \rbrace\quad \text{and}\quad	G(x)\subset \lbrace\bz\rbrace+\mu \ball_{Y} \Vert x-\bx \Vert\quad\text{for each}\quad x \in \ball_X[\bx,\beta],
		\end{eqnarray*}
	we have that the mapping $G+F$ is strongly metrically subregular at $(\bx,\by+\bz)$ with the constant $\kappa'$ and  the \nn $\ball_{X}[\bx,\beta]$. 
	\end{proposition}
	\begin{proof}
			Fix any $\mu>0$, $\beta>0$, and any mapping $G$ as in the conclusion.  Fix any  $x\in \ball_X[\bx, \gamma]$.
			
			If $(G(x)+F(x)) =\emptyset$, we are done.  If not, fix any $z\in G(x)$, then
			$$
			\Vert z-\bz\Vert \leq \mu \Vert x-\bx\Vert.
			$$
Then
			\begin{eqnarray*}
				\Vert x-\bx\Vert&\leq& \kappa \dist(\by,F(x))\leq \kappa \dist(\by+\bz,z+F(x))+\kappa\Vert z-\bz\Vert\\
				& \leq &\kappa \dist(\by+\bz,z+F(x))+\kappa\mu \Vert x-\bx\Vert.
			\end{eqnarray*}
			Taking into account that $\kappa\mu<1$ and $\tfrac{\kappa}{1-\kappa\mu}<\kappa'$ and that $z$ is fixed arbitrary, we obtain
			\begin{eqnarray*}
			\Vert x-\bx\Vert&\leq&\kappa/(1-\kappa \mu) \dist(\by+\bz,G(x)+F(x))\leq\kappa' \dist(\by+\bz,G(x)+F(x)).
		\end{eqnarray*}
		
	\end{proof}
We will now demonstrate that subregularity around each point of a compact set implies uniform subregularity. In other words, it is possible to find the same constant and \nn for all points within this set.

	\begin{theorem}
		\label{thmParametricStabilitySubregularity}
		Let $P$ be a  metric space and let $(X, \Vert \cdot\Vert)$ and $(Y,\Vert \cdot \Vert)$ be Banach spaces and let $\Omega \subset P\times X$ be a compact set. Consider a set-valued mapping $F:X\tto Y$ and a continuous single-valued mapping $f: P\times X\longrightarrow Y$ such that for each $({t},\bx)\in \Omega$ we have:
		\begin{enumerate}
			\item[\rm (i)]  the mapping $X\ni x \longmapsto G_{{t}}(x):= f({t},x)+F(x)$  is strongly metrically subregular around $(\bx,0)$;
			\item[\rm (ii)]  for each $\mu>0$ there is $\alpha>0$ such that for each $x, u \in \ball_X[\bx,\alpha]$ and each $s \in\ball_P[{t},\alpha]$ we have
			\begin{eqnarray*}
				\Vert f(s,u)-f({t},u)-(f(s,x)-f({t},x))\Vert \leq \mu \Vert x-u\Vert.
			\end{eqnarray*}
		\end{enumerate}
		Then:
		\begin{enumerate}
					 \item[\rm (iii)] there are positive constants $a,  b$, and $\kappa$ such that for each $({t},\bx)\in \Omega$ the mapping $G_{{t}}$ is strongly metrically subregular around $(\bx,0)$ with the constant $\kappa$ and \nns\,$\ball_X[\bx, a]$ and $\ball_Y[0, b]$;
					 \item[\rm (iv)] there are $\kappa'>0$ and $c>0$ such that for each $(t,x)\in \Omega$ the mapping $G_t$ is strongly metrically subregular \textbf{at} $(x,0)$ with the constant $\kappa'$ and the \nn $\ball_{X}[x,c]$.
		\end{enumerate}
	\end{theorem}
	\begin{proof} We are showing only {\rm (iii)}. The proof of (iv) follows similarly from Proposition \ref{propStability}, see also \cite[Theorem 2.6]{CPR2019}.
Fix any $({t},\bx) \in \Omega$. Find positive $a, b$, and $\kappa$, such that for each $(x,y)\in \big(\ball_{X}[\bx,a]\times \ball_{Y}[0, b]\big)\cap\gph\,G_{{t}}$ there is $r>0$ such that for each $u\in \ball_X[x,r]$ we have
\begin{eqnarray*}
	\Vert u-x\Vert\leq \kappa \dist(y,G_{\bar{t}}(u)). 
\end{eqnarray*}
Let $\mu:=1/(2\kappa)$ and $\kappa':= 3\kappa$. Then $\kappa \mu <1$ and $\kappa'>2\kappa=\kappa/(1-\kappa \mu)$. Find $\alpha\in \left(0,  b/(2\mu)\right)$ such that each $x, u \in \ball[\bx,2\alpha]$ and each $s \in\ball_P[{t},\alpha]$ we have
\begin{eqnarray*}
	\Vert f(s,u)-f({t},u)-(f(s,x)-f({t},x))\Vert \leq \mu \Vert x-u\Vert.
\end{eqnarray*}
Let $\beta:=b/4$. Then $2\beta +\mu \alpha<b/2+b/2=b.$ Since $f$ is continuous, there is $r'\in (0,\alpha/2]$ such that
\begin{eqnarray*}
	\Vert f(s,\bx)-f({t},\bx)\Vert\leq \beta\quad \text{for each}\quad s \in\ball_P[{t},r'].
\end{eqnarray*}
 Fix any $(s, x) \in\big(\ball_P[{t}, r']\times \ball_{X}[\bx, r']\big)\cap \Omega$.  Define a mapping $g: X\longrightarrow Y$ such that
\begin{eqnarray*}
	g(u):= f(s,u)-f({t},u)\quad\text{for}\quad u\in X.
\end{eqnarray*}
Then $G_{s}=g+G_{{t}}$ and for  each $x, u \in \ball[\bx,2\alpha]$ we have
	\begin{eqnarray*}
		\Vert g(x)-g(u)\Vert \leq \mu \Vert x-u\Vert\quad\text{and}\quad \Vert g(\bx)\Vert \leq \beta.
	\end{eqnarray*}
Theorem \ref{thmStabilitySubregularity}, with $G:=G_{{t}}$ and $\by:=0$, implies that for each $y\in \ball_Y[0,\beta]$ and each $x\in (g+G_{{t}})^{-1}(y)\cap \ball_X[\bx, \alpha]$  there is $r\in (0,\alpha]$ such that for each $u\in\ball_X[x,r]$ and  each   $ v\in  (g+G_{{t}})(u)\cap \ball_X[0, \beta]$   we have
\begin{eqnarray*}
	\Vert u-x\Vert\leq \kappa' \Vert y-v\Vert.
\end{eqnarray*}
We are showing that for each $(x,y) \in\big(\ball_{X}[\bx,\alpha/3]\times\ball_{Y}[0,\beta/3]\big)\cap \gph\,G_s$ there is $r>0$ such that for each $u\in \ball_{X}[x, r]$ we have
\begin{eqnarray*}
	\Vert u-x\Vert\leq \kappa'\dist(y,G_s(u)).
\end{eqnarray*}
Fix any such $(x, y)$ and find a corresponding $r\in (0,2\kappa'\beta/3]$ as in the claim and fix any $u\in \ball_{X}[x, r]$.  Thus $x\in (g+G_t)^{-1}(y)\cap \ball_{X}[\bx,\alpha]$. Fix any $v\in G_{s}(u)$. If $\Vert v\Vert\leq \beta$, using the claim, we get $\Vert u-x\Vert\leq \kappa' \Vert y-v\Vert$.

 If $\Vert v\Vert> \beta$, then $\Vert y-v\Vert\geq \Vert v\Vert-\Vert y\Vert> \beta-\beta/3=2/3\beta$  and so
 \begin{eqnarray*}
 	\Vert u - x \Vert\leq r\leq 2\kappa'\beta/3< \kappa'\Vert y-v\Vert.
 \end{eqnarray*}
 To sum up, we show that for each $({t},\bx)\in \Omega$ there are constants $\kappa'>0$, $\alpha>0,$ $\beta>0$, and $r'\in (0,\alpha/2)$ such that for each $(s, u) \in\big(\ball_P[{t}, r']\times \ball_{X}[\bx, r']\big)\cap \Omega$  and each $(x,y) \in\big(\ball_{X}[u,\alpha]\times\ball_{Y}[0,\beta]\big)\cap \gph\,G_{{s}}$ there is $r>0$ such that for each $v\in \ball_{X}[x, r]$ we have
 \begin{eqnarray*}
 \Vert v-x\Vert \leq \kappa'(y,G_s(v)).
 \end{eqnarray*}
 Note that  $\ball_{X}[u,r']\subset\ball_{X}[\bx,\alpha]$, then $G_s$ is strongly metrically subregular around $(x,0)$  with the constant $\kappa'$ and \nns $\ball_{X}[x,\alpha]$ and $\ball_{Y}[0, \beta]$.
 
 So $\kappa'$, $\alpha,$ $\beta$, and $r'$ depends only on the choice $(\bar{t},\bx)\in \Omega$. Then from open covering $\cup_{z=(t,x)\in\Omega} \big(\ball_P(t,r_z')\times\ball_{X}(x, r_z'))$ of compact set $\Omega$ find a finite subcovering $\mathcal{O}_i:= \big(\ball_P(t_i,r_{i}')\times\ball_{X}(x_i, r_{i}'))$ for $i=1,2,3, \dots, N$.
 Let $a:=\min\lbrace \alpha_{i}: i=1,2,3, \dots, N\rbrace $, $b:=\min\lbrace \beta_i: i=1,2,3, \dots, N\rbrace $, and $\kappa:=\max\lbrace \kappa'_i: i=1,2,3, \dots, N\rbrace.$ For any $(t,x)\in \Omega$ there is an index $i\in \lbrace 1,2,3,\dots, N\rbrace$ such that $(t,x)\in\mathcal{O}_i$. Hence the mapping $G_t$ is strongly metrically subregular around $(x,0)$ with the constant $\kappa$ and \nns $\ball_{X}[x,a]$ and $\ball_{Y}[0, b]$.
\end{proof}

Let us note, when $f$ is continuously differentiable, then the condition {\rm (ii)} is satisfied.

Similarly to the previous result, strong metric subregularity at each point of a compact set implies uniform strong metric subregularity on the entire set

\begin{theorem}
\label{thmParametricStabilitySubregularityAt}
Let $P$ be a  metric space and let $(X, \Vert \cdot\Vert)$ and $(Y,\Vert \cdot \Vert)$ be Banach spaces and let $\Omega \subset P\times X$ be a compact set. Consider a set-valued mapping $F:X\tto Y$ and a continuous single-valued mapping $f: P\times X\longrightarrow Y$ such that for each $({t},\bx)\in \Omega$ we have:
\begin{enumerate}
	\item[\rm (i)]  the mapping $X\ni x \longmapsto G_{{t}}(x):= f({t},x)+F(x)$  is strongly metrically subregular \textbf{at} $(\bx,0)$;
	\item[\rm (ii)]  for each each $\mu>0$ there is $\alpha>0$ such that for each $x, u \in \ball_X[\bx,\alpha]$ and each $s \in\ball_P[{t},\alpha]$ we have
	\begin{eqnarray*}
		\Vert f(s,u)-f({t},u)-(f(s,x)-f({t},x))\Vert \leq \mu \Vert x-u\Vert.
	\end{eqnarray*}
\end{enumerate}
Then there are $\kappa'>0$ and $c>0$ such that for each $(t,x)\in \Omega$ the mapping $G_t$ is strongly metrically subregular \textbf{at} $(x,0)$ with the constant $\kappa'$ and the \nn $\ball_{X}[x,c]$.
\end{theorem}
\begin{proof}
	The proof is similar to the proof of Theorem \ref{thmParametricStabilitySubregularityAt}, but instead of using Theorem \ref{thmStabilitySubregularity}, we apply Proposition \ref{propStability}.
\end{proof}

The following statement ensures that uniform strong metric regularity is maintained along continuous paths. This means that as one follows a continuous path within the domain, the property of strong metric subregularity remains consistent and uniform, providing stability and predictability in the behaviour of the system

	\begin{theorem}
		\label{thmStabilitySubreg}
		Let $T>0$ be fixed and let $(X,\Vert\cdot\Vert)$ and $(Y,\Vert \cdot\Vert)$ be Banach spaces. Consider a set-valued mapping $F:X\tto Y$ and a continuous single-valued mapping $f:[0,T]\times X\longrightarrow Y$, and two continuous mappings $\varphi:[0,T]\longrightarrow X$ and $p:[0,T]\longrightarrow Y$ such that 
		\begin{enumerate}
			\item[\rm (i)]  for each $t\in [0,T]$ the mapping $X\ni x\longmapsto G_t(x):=f(t,x)+F(x)$ is strongly metrically subregular around $(\varphi(t),p(t))$;
			\item[\rm (ii)]  for each $t\in [0,T]$ and each $\mu>0$ there is $\delta>0$ such that for each $x,u\in \ball_{X}[\varphi(t),\delta]$ and each $s\in (t-\delta, t+\delta)$ we have
			\begin{eqnarray*}
				\Vert f(s,u)-f(t,u)-(f(s,x)-f(t,x))\Vert \leq \mu \Vert x-u\Vert.
			\end{eqnarray*}
		\end{enumerate}
		Then: \begin{enumerate}
			\item[\rm (iii)] there are positive constants $a, b$, and $\kappa$ such that for each $t\in [0,T]$ the mapping $G_t$ is strongly metrically subregular around $(\varphi(t),p(t))$ with the constant $\kappa$ and \nns $\ball_X[\varphi(t),a]$ and $\ball_{Y}[p(t),b]$;
			\item[\rm (iv)] there are $c>0$ and $\kappa'>0$ such that for each $t\in [0,T]$ the mapping $G_t$ is strongly metrically subregular at $(\varphi(t),p(t))$ with the constant $\kappa'$ and \nn $\ball_X[\varphi(t), c]$.
			\end{enumerate}
	\end{theorem}
	\begin{proof}
For {\rm (iii)}, apply Theorem \ref{thmParametricStabilitySubregularity}, and for {\rm (iv)}, apply Theorem \ref{thmParametricStabilitySubregularityAt}, both with $P := [0, T] \times Y$, the compact set $\Omega := \bigcup_{t \in [0, T]} (t, p(t), \varphi(t))$, and the function $f(t, x) := f(p, x) - y$ for $t = (p, y) \in P$ and $x \in X$. 
	\end{proof}

 \section{Uniform semismooth* path-following method}
 
 In this section, we delve into the semismooth* path-following method, designed to address a  problem of finding a mapping $x:[0,T]\longrightarrow\mathbb{R}^n$ that satisfies the inclusion \eqref{eqInclusionParametric},
 where $F:\mathbb{R}\times \mathbb{R}^n\tto \mathbb{R}^n$ is a set-valued mapping and $T>0$. Particularly, we focus on scenarios, where $F$ exhibits uniform semismoothness$^*$ properties.
 
 The challenge of solving \eqref{eqInclusionParametric} is non-trivial due to the inherent complexities associated with the behaviour of set-valued mappings. To systematically address this, for each time point $t\in [0,T]$, we define a set-valued mapping $G_t:\mathbb{R}^n\tto \mathbb{R}^n$ and a corresponding solution mapping $S:\mathbb{R}\tto \mathbb{R}^n$ of \eqref{eqInclusionParametric}, expressed as
 \begin{eqnarray*}
 	G_t(x) := F(t,x)\quad \text{for }\quad x\in \mathbb{R}^n\quad\text{and}\quad	S(t) := \lbrace x\in \mathbb{R}^n : 0 \in F(t,x)\rbrace.
 \end{eqnarray*}
 
The following lemma provides an approximation estimate for $(A, B) \in \mathcal{A}_{\text{reg}} G_t (x, y)$ under the assumption of uniform semismoothness$^*$.
 \begin{lemma}
 	\label{lemMatrices}
 Assume that for each $\varepsilon>0$ there is $\beta>0$ such that for each $t\in [0,T]$ and each $ (u,y)\in\left(\ball_{\mathbb{R}^n}[x(t),\beta]\times \ball_{\mathbb{R}^n}[0,\beta]\right)\cap \gph\,G_t$ and each $ ( y^*, x^*)\in  {D}^*G_t( u,y)$ we have
 \begin{eqnarray}
 	\vert \left< x^*,u-x(t)\right>-\left<y^*,y\right> \vert\leq \varepsilon \Vert ( x^*,y^*)\Vert \Vert ( u,y)-( x(t),0)\Vert.
 \end{eqnarray}
 	Then for each $\varepsilon>0$ there is $\beta>0$ such that for each $t\in [0,T]$ and each $(u,y)\in\left(\ball_{\mathbb{R}^n}[x(t),\beta]\times \ball_{\mathbb{R}^n}[0,\beta]\right)\cap \gph\,G_t$ and for each $(A,B)\in \mathcal{A}_{\text{reg}}\,G_t({u},{y}) $ we have
 	\begin{eqnarray*}
 		\Vert (u-A^{-1}By-x(t)\Vert \leq \varepsilon \Vert A^{-1}\Vert \Vert (A\,\arrowvert\, B)\Vert_F 	\Vert( u,y) -(x(t),0)\Vert.
 	\end{eqnarray*}
 	
 \end{lemma}	
 \begin{proof}We follow the proof from \cite[Proposition 4.3]{GO2022}.
 	Fix any $\varepsilon > 0$ and choose a corresponding $\beta > 0$ such that $\beta$ meets the conditions of the lemma. Consider any $t \in [0, T]$ and any $(u, y)\in (\ball_{\mathbb{R}^n}[x(t), \beta] \times \ball_{\mathbb{R}^n}[0, \beta]) \cap \gph\, G_t$. Fix any $(A,  B) \in \mathcal{A}_{\text{reg}}\,G_t( u, y)$.
 	
 	According to the definition of $\mathcal{A}_{\text{reg}}\, G_t( u, y)$, for each pair $(y^*_i, x^*_i) \in {D}^*G_s( u, y)$, the $i$-th component of the vector $A(u - x(t)) - By$ is given by $\langle x^*_i, u - x(t) \rangle - \langle y^*_i, y \rangle$. The absolute value of this expression can be bounded by $\varepsilon \| (x^*_i,  y^*_i) \| \| (u, y) - ( x(t), 0) \|$.
 	
 	The Euclidean norm of the vector formed by the norms $\| (x^*_i, y^*_i) \|$ is equivalent to the Frobenius norm of the matrix $(A \,|\,  B)$. Thus, we have
 	$$
 	\| A(u - x(t)) - By \| \leq \varepsilon \| (A \,|\, B) \|_F \| ( u, y) - ( x(t), 0) \|.
 	$$
 	Then
 \begin{eqnarray*}
 	&&	\| (u - A^{-1}By- x(t) \|	 \leq   \| A^{-1} \|	\| A(u - x(t)) - By \|\\
 	&\leq& \varepsilon \| A^{-1} \| \| (A \,| \, B) \|_F \| ( u, y) - ( x(t), 0) \|.
 \end{eqnarray*}
 \end{proof}

 We propose the semismooth$^*$ path-following method for solving \eqref{eqInclusionParametric}, based on the semismooth$^*$ method introduced in \cite{GO2021}. This method leverages the inherent properties of semismoothness$^*$ to efficiently track solutions over a given interval.
 
 To implement this, we begin by defining a uniform grid over the interval $[0, T]$, partitioning it into $N$ equal segments. Each segment corresponds to a discrete time step, $ t_k $, where $ k = 0, 1, \ldots, N $ and $ t_k = kh $ with $ h = T/N $.  At each  time point $ t_{k+1} $, we apply one step of the semismooth$^*$ method. Starting from an initial point $ x_k $, we compute the next iterate $ x_{k+1} $.
 \begin{algorithm}\label{algPathFollowingUniform}
 	\begin{enumerate}
 		\item 	Choose $N\in \mathbb{N}$ and $x_0$  close to $x(0)$, set the step $h:={T}/{N}$, set $t_k:=k h$ for $k= 0, 1,2,\dots, N$, and the counter $k:=0$.
 		\item If $0\in F(t_{k+1}, x_k)$, then set $x_{k+1}:=x_k$ and if $k<N-1$ then $k:=k+1$ and go  to 2; otherwise stop the algorithm.
 		\item \textbf{Approximation step}: Compute
 		$(\hat{x}_{k+1},\hat{y}_{k+1})\in \gph\,G_{t_{k+1}}$
 		such that  
 		\begin{eqnarray*}
 			\mathcal{A}_{\text{reg}}\, G_{t_{k+1}} (\hat{x}_{k+1},\hat{y}_{k+1})\neq \emptyset\quad\text{and}\quad	\Vert (\hat{x}_{k+1},\hat{y}_{k+1})-(x(t_{k}),0)\Vert\leq \eta \Vert x_{k}-x(t_{k})\Vert.
 		\end{eqnarray*}
 		\item \textbf{Newton step:} Select $n\times n$ matrices $(A_{k+1}, B_{k+1})\in \mathcal{A}_{\text{reg}}\, G_{t_{k+1}}
 	( \hat{x}_{k+1},\hat{y}_{k+1})$ and compute the new iterate via $x_{k+1}=\hat{x}_{k+1}-A^{-1}_{k+1}B_{k+1} \hat{y}_{k+1} $.
 		\item If $k=N-1$,  stop the algorithm.
 		\item Set $k:=k+1$ and go to 2.
 	\end{enumerate}
 \end{algorithm}

The following theorem provides a bound on the grid error of our semismooth$^*$ path-following method under the assumption of uniform semismoothness$^*$. This error analysis is crucial as it quantifies the accuracy of the method when applied to the problem \eqref{eqInclusionParametric}. By establishing the conditions under which the grid error remains controlled, we can ensure that the iterates produced by the algorithm remain close to the exact solution trajectory. This theorem leverages the Lipschitz continuity of the solution mapping and certain properties of the set-valued mapping $F$ to derive a bound on the error that depends on the grid size and other parameters.

 \begin{theorem}\label{thmGridErrorUniform} Consider a set-valued mapping $F :\mathbb{R}\times \mathbb{R}^n\tto \mathbb{R}^n$, with a closed graph, positive numbers  $T, \eta, \ell, \alpha$ and $\kappa$, and a mapping $x:[0,T] \longrightarrow \mathbb{R}^n$ satisfying \eqref{eqInclusionParametric}. Suppose that $x(\cdot)$ is Lipschitz continuous on $[0,T]$ with the constant $\ell$.
 	Moreover, assume that for each $\varepsilon>0$ there is $\beta>0$ such that for each $t\in [0,T]$ and each $ (u,y)\in\left(\ball_{\mathbb{R}^n}[x(t),\beta]\times \ball_{\mathbb{R}^n}[0,\beta]\right)\cap \gph\,G_t$ and each $ ( y^*, x^*)\in  {D}^*G_t( u,y)$ we have
 	\begin{eqnarray}
 		\label{eqSemismoothnessCon}
 		\vert \left< x^*,u-x(t)\right>-\left<y^*,y\right> \vert\leq \varepsilon \Vert ( x^*,y^*)\Vert \Vert ( u,y)-( x(t),0)\Vert
 	\end{eqnarray}
 	and $\mathcal{G}_{G_t,x(t)}^{\eta, \kappa}(u)\neq \emptyset$ for each $u \in \ball_{\mathbb{R}^n}[x(t),\alpha]$ with $u\notin S(t)$.  Then there is $N_0\in \mathbb{N}$ such that for each $N>N_0$ and each $x_0\in \ball_{\mathbb{R}^n}[x(0), \ell h]$, where $h:=T/N$, there are points $x_1,x_2,\dots, x_{N-1}, x_{N}\in \mathbb{R}^n$, generated by Algorithm \ref{algPathFollowing}, satisfying
 	\begin{eqnarray*}
 		\max_{0\leq k \leq N}\Vert x_k-x(t_k)\Vert \leq\ell h.
 	\end{eqnarray*}
 \end{theorem}
 \begin{proof} Using Lemma \ref{lemMatrices} with $\varepsilon:=\tfrac{1}{2\eta\kappa}$ and find a corresponding $\beta>0$. Find $N_0\in \mathbb{N}$ such that $T/N_0 \leq \min\lbrace \beta/(2 \ell\eta), \alpha/\ell\rbrace$. Fix any $N>N_0$ and let $h:=\tfrac{T}{N}$.	
 	Fix any $x_0\in \ball_{\mathbb{R}^n}[x(t_0), \ell h]$.

 	Assume that for some $k\in\lbrace 0, 1,2,\dots,N-1 \rbrace$ we have
 	\begin{eqnarray*}
 		\Vert x_{k}-x(t_{k})\Vert \leq \ell h.
 	\end{eqnarray*}
 	Clearly, $  		\Vert x_{k}-x(t_{k})\Vert \leq \ell h\leq \alpha$. Hence there are $(\hat{x}_{k+1},\hat{y}_{k+1}, A_{k+1},  B_{k+1}) \in \mathcal{G}_{F,x(t_{k+1})}^{\eta, \kappa}(x_k)$ and we have
 	\begin{eqnarray*}
 		\Vert (\hat{x}_{k+1},\hat{y}_{k+1})-(x(t_{k+1}),0)\Vert\leq\eta\Vert x_k -x(t_{k+1})\Vert \leq\eta\Vert x_k -x(t_k)\Vert+\eta\Vert x(t_k) -x(t_{k+1})\Vert\leq 2 \eta  \ell h\leq\beta.
 	\end{eqnarray*}
 Then
\begin{eqnarray*}
	\Vert x_{k+1}-x(t_{k+1})\Vert &=& 	\Vert \hat{x}_{k+1}-A^{-1}_{k+1}B_{k+1} \hat{y}_{k+1}-x(t_{k+1})\Vert\\
	&\leq& \varepsilon \Vert A^{-1}_{k+1}\Vert \Vert (A_{k+1}\,\arrowvert\,  B_{k+1})\Vert_F  \Vert \Vert(\hat{x}_{k+1},\hat{y}_{k+1})-(x(t_{k+1}),0))\Vert
	\\ &\leq& \varepsilon \eta \Vert A^{-1}_{k+1}\Vert \Vert (A_{k+1}\,\arrowvert\, B_{k+1})\Vert_F 	\Vert x_k-x(t_{k+1})\Vert  
	\leq \tfrac{1}{2}	\Vert x_k-x(t_{k+1})\Vert\\ &\leq& 	\tfrac{1}{2}	\Vert x_k-x(t_k)\Vert+\tfrac{1}{2}	\Vert x(t_k)-x(t_{k+1})\Vert\leq \ell h.
\end{eqnarray*}
 \end{proof}
The following result is derived from \cite[Proposition 2.8]{BM2024}, which provides a  result for the existence of matrices that are instrumental in ensuring the regularity conditions required for our analysis. For further details and a broader context, see also \cite[Theorem 3.1]{BM2021}.

\begin{proposition} \label{propMatricesExistence} Let $F: \mathbb{R}^n \rightrightarrows \mathbb{R}^n$ be a set-valued mapping, with a closed graph, and $\kappa>0$ be given. Assume that $F$ is strongly metrically subregular at $(\hat{x}, \hat{y}) \in \gph\,F$ with a constant $\kappa>0$. Then there is a matrix $B \in \mathbb{R}^{n \times n}$ with $\|B\| \leq \kappa$ such that $(I_n, B) \in \mathcal{A}_{\text{reg}}\,F(\hat{x}, \hat{y})$ and $e_i\in D^*F(\hat{x},\hat{y})((B)_i^T)$ for each $i=1,2,\dots,n$. 
\end{proposition}
In order to apply the semismooth$^*$ path-following method effectively, it is essential to confirm that the set-valued mapping $G_t$ satisfies certain regularity conditions. These conditions include strong metric subregularity and semismooth$^*$ properties, which play a pivotal role in ensuring the convergence and stability of the algorithm. The following proposition establishes the existence of these properties for the set-valued mapping $G_t$ under the given assumptions.

Under the assumption of strong metric subregularity around the reference point, we demonstrate that the set $\mathcal{G}_{F, x(t)}^{\eta, \kappa}(u)$ is nonempty. This validation confirms the stability and robustness of solution paths, ensuring the regularity and feasibility of the solutions throughout the domain.
\begin{proposition}  Consider a set-valued mapping $F:  \mathbb{R}^n\tto \mathbb{R}^n$,  with a closed graph, a single-valued mapping $f: \mathbb{R}\times\mathbb{R}^n\longrightarrow \mathbb{R}^n$, a positive number  $T,$ and a mapping $x:[0,T] \longrightarrow \mathbb{R}^n$ satisfying  \eqref{eqInclusionParametric} with $F:=f+F$. Suppose that for each $t\in [0, T]$   the mapping $G_t$ is strongly metrically subregular around $(x(t),0)$.
	Then  there are $\kappa>0$, $\alpha>0$, and $\ell>0$ such that for each $t\in [0,T]$ and each $\eta\geq 1$ we have $\mathcal{G}_{f+F,x(t)}^{\eta, \sqrt{n(1+\kappa^2)}}(u)\neq \emptyset$  each $u \in \ball_{\mathbb{R}^n}[x(t),\alpha]$. 
\end{proposition}

\begin{proof}
	Theorem \ref{thmStabilitySubregularity}, with $(\varphi,p):=(x(\cdot),0)$, implies that there are $\kappa>0$ and $\alpha>0$ such that for each $t\in [0,T]$ the mapping $G_t$ is strongly metrically subregular around $(x(t),p(t))$ with the constant $\kappa$ and \nns $\ball_{\mathbb{R}^n}[x(t),\eta\alpha]$ and $\ball_{\mathbb{R}^n}[0,\eta\alpha]$.
	
Fix any $t\in [0,T]$ and any $\eta\geq 1$.  	Fix any $(\hat{x},\hat{y})\in\ball_{\mathbb{R}^n}[x(t),\eta\alpha]\times\ball_{\mathbb{R}^n}[0,\eta \alpha]\cap \gph\,G_t$, then, by Proposition \ref{propMatricesExistence}, with $F:=G_t$,   there is a matrix $B \in \mathbb{R}^{n \times n}$, with $\|B\| \leq \kappa$, such that $(I_n, B) \in \mathcal{A}_{\text{reg}}\, G_t(\hat{x}, \hat{y})$ and $e_i\in D^*G_t(\hat{x},\hat{{y}})((B)_i^T)$ for each $i=1,2,\dots, n$. Then  $\Vert(I_n,  B)\Vert_F\leq \sqrt{n (1+ \kappa^2)}$.

	Fix any $u \in\ball_{\mathbb{R}^n}[x(t), \alpha]$. Find  $(\hat{x},\hat{y})\in \gph\,G_t$ such that $\Vert(\hat{x}-x(t),\hat{y})\Vert\leq \eta\Vert u-x(t)\Vert\leq \eta\alpha.$
	Then, by the above, there is $B\in \mathbb{R}^{n\times n}$ such that $(I_n,B)\in \mathcal{A}_{\text{reg}}\, G_t(\hat{x},\hat{y})$  and $\Vert(I_n,  B)\Vert_F\leq \sqrt{n(1+\kappa^2)}$, therefore $\mathcal{G}_{f+F,x(t)}^{\eta, \sqrt{n(1+\kappa^2)}}(u)\neq \emptyset$.
\end{proof}

\section{Point-wise semismooth* path-following method}
In this section, we focus on the  semismooth$^*$  path-following method for the problem \eqref{eqInclusionParametric}.   This study is particularly focused on scenarios where $F$ lacks uniform semismooth$^*$ property.
 The algorithms need the following: a
point $ (t,x,y) \in \gph\,F $ and  sets $\mathcal{B}_{\text{reg}}\, F(t, {x},{y}) \subset \mathbb{R}^{n\times n}\times \mathbb{R}^n\times \mathbb{R}^{n\times n}$ and $\mathcal{Q}_{F, \bar{x}}^{\eta, \kappa}(t, x)\subset \mathbb{R}^n\times \mathbb{R}^n\times\mathbb{R}^{n\times n}\times \mathbb{R}^n\times\mathbb{R}^{n\times n}$ such that
$$
\mathcal{B}_{\text{reg}}\, F(t,{x},{y}):=\left\lbrace (A,V,B):  ( (B)_i^T,V_i, (A)_i^T)\in \gph\,D^* F(t, {x},{y}),  i \in \lbrace 1,2,\dots,n  \rbrace \text{ and } A^{-1}\text{ exists}\right\rbrace
$$	
and
\begin{eqnarray*}
	\mathcal{Q}_{F, \bar{x}}^{\eta, \kappa}(t,x):=\lbrace(\hat{x}, \hat{y}, A,V,B): \Vert(\hat{x}-\bar{x},\hat{y})\Vert \leq \eta \Vert x-\bar{x} \Vert, (A,V, B)\in\mathcal{B}_{\text{reg}}\, F(t,\hat{x},\hat{y}), \Vert A^{-1}\Vert \Vert (A\mid V\mid B)\Vert_F \leq \kappa \rbrace,
\end{eqnarray*}
where $\bx \in S(t) $, and  $ \eta$ and $\kappa $ are positive numbers.

 The motivation for this exploration is drawn from the work in \cite{LNK2018}, which provides a theoretical underpinning for handling semismoothness in dynamic systems.
 
To effectively tackle the challenges presented by the semismooth$^*$
path-following method under conditions of non-uniform semismoothness, it is essential to adapt and refine our analytical tools. For our purposes, we need to define sets similar to ones from the previous section but tailored to address the specificities and complexities of the non-uniform semismoothness$^*$ environment.

The version of Lemma \ref{lemMatrices} for this case reflects the nuanced understanding required when dealing with mappings that are semismooth$^*$, but lack of uniformity in their properties. 
\begin{lemma}
	\label{lemMatrices1}
	Assume for each $t\in [0,T]$  the mapping $F$ is semismooth$^*$ at $(t,x(t), 0)$.
	Then for each $\varepsilon>0$  and for each $t\in [0,T]$ there is $\beta>0$ such that for each $(s, u,y)\in\left(\ball_{\mathbb{R}}[t,\beta]\times\ball_{\mathbb{R}^n}[x(t),\beta]\times \ball_{\mathbb{R}^n}[0,\beta]\right)\cap \gph\,F$ and for each $(A, V, B)\in {\mathcal{B}_{\text{reg}}}\,F(s, {u},{y}) $ we have
	\begin{eqnarray*}
		\Vert (u-A^{-1}By+A^{-1}V(s-t))-x(t)\Vert \leq \varepsilon \Vert A^{-1}\Vert \Vert (A\,\arrowvert\,V\,\arrowvert\,B)\Vert_F 	\Vert(s, u,y) -(t, x(t),0)\Vert.
	\end{eqnarray*}
\end{lemma}	
\begin{proof}The proof comes directly from \cite[Proposition 4.3]{GO2021}.
\end{proof}

Now, we introduce the algorithm for semismooth$^*$ path-following method.
It begins with an initial setup close to a potential solution, progressively refining this starting point through calculated step. These steps are guided by the semismooth$^*$ characteristics of the mapping 
$F$, ensuring a robust adaptation to its behaviour. This approach provides a systematic method to tackle its variability and complexity.

\begin{algorithm}\label{algPathFollowing}
	\begin{enumerate}
			\item 	Set $t\in [0,T)$, $s\in (t,T]$, and   $\bx\in\mathbb{R}^n$.
			\item If $0\in F(s, \bx)$, then set $\bar{u}:=\bx$ and end the algorithm. 
			\item \textbf{Approximation step}: Compute $(\hat{x},\hat{y})$ such that
			$(s, \hat{x},\hat{y})\in \gph\,F$
			such that  
			\begin{eqnarray*}
					{\mathcal{B}_{\text{reg}}}\, F (s, \hat{x},\hat{y})\neq \emptyset\quad\text{and}\quad	\Vert (\hat{x},\hat{y})-( x(s),0)\Vert\leq \eta \Vert \bx-x(s)\Vert.
			\end{eqnarray*}
			\item \textbf{Newton step:} Select  $(A,V, B)\in \mathcal{B}_{\text{reg}}\,F(s, \hat{x},\hat{y})$, $A$ and $B$ are $n\times n$ matrices and compute  $\bar{u}:=\hat{x}-A^{-1}B\hat{y}+A^{-1} V (s-t)$.
		\end{enumerate}
\end{algorithm}

Having delineated the steps of the semismooth$^*$  path-following algorithm, we now transition to a rigorous validation of this approach through a formal theorem. This theorem provides sufficient conditions guaranteeing the error estimate under the conditions specified by the semismooth$^*$ properties of the set-valued mapping $F$.

\begin{theorem}\label{thmSemismoothMethod} Consider a set-valued mapping $F :\mathbb{R}\times \mathbb{R}^n\tto \mathbb{R}^n$,  with a closed graph, positive numbers  $T, \eta, \ell,$ and $\kappa$ and a mapping $x:[0,T] \longrightarrow \mathbb{R}^n$. Suppose that $x(\cdot)$ is Lipschitz continuous on $[0,T]$ with the constant $\ell$ and satisfies \eqref{eqInclusionParametric}.
	Moreover, assume that  there is $\alpha>0$ such that for  each $t\in [0,T]$ we have $\mathcal{Q}_{F, x(t)}^{\eta, \kappa}(t,u)\neq \emptyset$ for each $u \in \ball_{\mathbb{R}^n}[x(t),\alpha]$ with $u\notin
	 S(t)$. Suppose  that there are $t\in [0,T)$ and $\beta>0$ such that for each $( s, u,y)\in\left(\ball_{\mathbb{R}}[t,\beta]\times\ball_{\mathbb{R}^n}[x(t),\beta]\times \ball_{\mathbb{R}^n}[0,\beta]\right)\cap \gph\,F$ and for each $(A,V,B)\in {\mathcal{B}_{\text{reg}}}\,F(s,{u},{y}) $ we have
	 \begin{eqnarray*}
	 	\Vert u-A^{-1}By+A^{-1}V(s-t)-x(t)\Vert \leq 1/(2\eta\kappa) \Vert A^{-1}\Vert \Vert (A\,\arrowvert\,V \arrowvert\ B)\Vert_F 	\Vert(s,u,y) -(t, x(t),0)\Vert.
	 \end{eqnarray*}  Suppose that we have $\bx\in \mathbb{R}^n$ such that $a:=\Vert \bx-x(t)\Vert <\min\lbrace \alpha,\beta/\eta\rbrace$. Then for each $s\in (t,  \min\lbrace T,t+\beta\rbrace]$ such that $\eta a+(\eta+1) \ell (s-t)\leq\beta$ and $a+\ell(s-t)\leq\alpha$ the point $\bar{u}\in\mathbb{R}^n$, generated by  Algorithm \ref{algPathFollowing},  satisfies
	 	\begin{eqnarray*}
	 	\Vert \bar{u}-x(s)\Vert \leq 1/2\max\lbrace \Vert \bx-x({t})\Vert,s-t\rbrace+ \ell  (s-t).
	 \end{eqnarray*}	
\end{theorem}
\begin{proof}  Let $\varepsilon:=1/(2\eta\kappa)$.
Fix any $s\in (t,  \min\lbrace T,t+\beta\rbrace]$ such that $\eta a+\ell(\eta+1)  (s-t)\leq\beta$ and $a+\ell(s-t)\leq\alpha$.  If $0\in F(s, \bx)$, then we finished.  Otherwise,
\begin{eqnarray}
	\Vert \bx-x(s)\Vert \leq \Vert \bx-x(t)\Vert +\Vert x(t)-x(s)\Vert\leq a+\ell(s-t)\leq\alpha;
\end{eqnarray}
hence  $\mathcal{Q}_{F,  x(s)}^{\eta, \kappa}(s, \bx)$ is nonempty. Then fix any $( \hat{x},\hat{y},A,V,  B) \in \mathcal{Q}_{F,  x(s)}^{\eta, \kappa}(s, \bx)$ and then we have $(s,\hat{x},\hat{y})\in\left(\ball_{\mathbb{R}}[t,\beta]\times\ball_{\mathbb{R}^n}[x(t),\beta]\times \ball_{\mathbb{R}^n}[0,\beta]\right)\cap \gph\,F$ since
	\begin{eqnarray*}
	&&\Vert (\hat{x},\hat{y})-( x(t),0)\Vert \leq \Vert (\hat{x},\hat{y})-( x(s),0)\Vert +\Vert x(t)-x(s)\Vert\leq \eta \Vert \bx-x(s)\Vert+\ell (s-t)\\
	&\leq&  \eta \Vert \bx-x(t)\Vert+ \eta \Vert x(t)-x(s)\Vert+\ell (s-t)\leq \eta\Vert \bx-x(t)\Vert+\ell(\eta+1) (s-t)\leq\beta.
\end{eqnarray*}
  Let $\bar{u}:=\hat{x}-A^{-1}
B\hat{y}+A^{-1}V(s-t)$, then
			\begin{eqnarray*}
			&&\Vert \bar{u}-x(s)\Vert=\Vert \hat{x}-A^{-1}
		B\hat{y}+A^{-1}V(s-t)-x(s)\Vert\leq \Vert \hat{x}-A^{-1}
		B\hat{y}+A^{-1}V(s-t)-x(t)\Vert\\
		& +&\Vert x(s)-x(t) \Vert\leq\varepsilon \eta\Vert A^{-1}\Vert \Vert (A\,\arrowvert\,V\arrowvert\, B)\Vert_F \Vert (s,\hat{x},\hat{y}) -(t, x(t),0)\Vert+\ell (s-t)\\
			&\leq &1/2\max\lbrace \Vert \bx-x({t})\Vert,s-t\rbrace+\ell(s-t).
			\end{eqnarray*}
	\end{proof}

Applying the idea for global convergence of the SCD Newton method from \cite{GOV2022}, we suggest the semismooth$^*$ path-following method which guarantees an estimated grid error on a non-uniform grid.

For this purpose, we assume that we have a method which is globally convergent in the sense that for each $t \in [0, T]$, it generates from an arbitrary starting point $u_0$ a sequence $(u_i)$ such that at least one accumulation point of $(u_i)$ satisfies \eqref{eqInclusionParametric} for that $t$. We suppose that this method is formally given by some mapping $\mathcal{T} : \mathbb{R} \times \mathbb{R}^n \rightarrow \mathbb{R}^n$, which computes from the iterate $u_i$ the next iterate by
\begin{eqnarray}
	\label{eqIteration}
u_{i+1} = \mathcal{T}(t, u_i).
\end{eqnarray}
Note that  $\mathcal{T}$ must depend on the problem \eqref{eqInclusionParametric}. For particular choices of $\mathcal{T}$ see \cite[Section 6.2]{GOV2022}.

\begin{algorithm}
	\label{algSemismoothWholeT}
	\begin{enumerate}
		\item Set $t_0:=0$, $x_0$ close to $x(t_0)$, and $k:=0$.
		\item Choose $\varepsilon>0$ and set $u_0:=x_k$ and $i:=0$. 
		If $\Vert u_0 - x(t_k) \Vert < \varepsilon$, go to 5.
		\item Compute
		\begin{eqnarray*}
			u_{i+1} := \mathcal{T}(t_k, u_i).
		\end{eqnarray*}
		\item If $\Vert u_{i+1} - x(t_k) \Vert > \varepsilon$, then $i := i + 1$ and go to 3.
		
		\item If $\Vert u_0 - x(t_k) \Vert < \varepsilon$, set $x_k := u_0$; otherwise set $x_k := u_{i+1}$. Choose $t_{k+1} \in (t_k, T]$.
		\item Using Algorithm \ref{algPathFollowing}, with $\bx:=x_{k}, t:=t_k$ and $s:=t_{k+1}$, get $ x_{k+1}:=\bar{u}.$ If $t_{k+1} = T$, then end the algorithm; otherwise set $k := k + 1$ and go to 2.
	\end{enumerate}
\end{algorithm}
\begin{theorem} Consider a set-valued mapping $F :\mathbb{R}\times \mathbb{R}^n\tto \mathbb{R}^n$, with a closed graph, positive numbers  $T, \eta, \ell,$ and $\kappa$ and a mapping $x:[0,T] \longrightarrow \mathbb{R}^n$. Suppose that $x(\cdot)$ is Lipschitz continuous on $[0,T]$ with the constant $\ell$ and satisfies \eqref{eqInclusionParametric}. Further, consider a single-valued mapping $\mathcal{T}:\mathbb{R}\times \mathbb{R}^n \longrightarrow \mathbb{R}^n$ such that for each $t\in [0,T]$  the sequence $(u_k)$, generated by \eqref{eqIteration} for any initial point $u_0 \in\mathbb{R}^n$, converges to $x(t)$.
	Moreover, assume that there is $\alpha>0$  such that for each $t\in [0,T]$ we have $\mathcal{Q}_{F, x(t)}^{\eta, \kappa}(t, u)\neq \emptyset$ for each $
	u \in \ball_{\mathbb{R}^n}[x(t),\alpha]$ with $u\notin
	S(t)$ and  $F$ is semismooth$^*$ at $(t,x(t),0)$. Then, for $t_0 = 0$ and $x_0 := x(0)$, there are neither an increasing sequence $(t_k)$ in $[0, T]$ nor finitely many numbers $t_0 < t_1 < t_2< \dots < t_{N-1} < t_N = T$, and a sequence $(x_k)$ or finitely many points $x_0, x_1, x_2, \dots, x_{N-1}, x_N$, both generated by Algorithm \ref{algSemismoothWholeT}, such that
	\begin{eqnarray}
\label{eqErrorGridEstimate}
	\Vert x_{k+1} - x(t_{k+1}) \Vert \leq \frac{1}{2} \max\lbrace\Vert x_k - x(t_k) \Vert, t_{k+1} - t_k\rbrace +  \ell (t_{k+1} - t_k)
\end{eqnarray}
	for each $k \in \mathbb{N}_0$ or each $k \in \{ 0, 1, 2, \dots, N-2, N-1 \}$.
\end{theorem}
\begin{proof} Denote \(\varepsilon := \frac{1}{2\eta \kappa}\).Assume that, for some $k\in\mathbb{N}$, we have $t_k\in [0,T]$ and $x_k\in \mathbb{R}^n$  such that
	$$
\Vert x_{k} - x(t_{k}) \Vert \leq \frac{1}{2} \Vert x_{k-1} - x(t_{k-1}) \Vert +  \ell (t_{k} - t_{k-1}).
$$
By Lemma \ref{lemMatrices1}, there exists \(\beta > 0\) such that
  $(s,u,y)\in\left(\ball_{\mathbb{R}}[t_k,\beta]\times\ball_{\mathbb{R}^n}[x(t_k),\beta]\times \ball_{\mathbb{R}^n}[0,\beta]\right)\cap \gph\,F$ we have that for each $(A,V,B)\in \mathcal{B}_{\text{reg}}\,F({t_k},{u},{y}) $ we have
\begin{eqnarray*}
	\Vert u-A^{-1}By-x(t_k)\Vert \leq \varepsilon \Vert A^{-1}\Vert \Vert (A\,\arrowvert\,V \,\arrowvert\,B)\Vert_F 	\Vert( s, u,y) -(t_k, x(t_k),0)\Vert.
\end{eqnarray*}
In Algorithm \ref{algSemismoothWholeT}, we choose $\varepsilon:=\min\lbrace \alpha,\beta/\eta\rbrace$.
If  $a:=\Vert x_{k}-x(t_k)\Vert <\min\lbrace \alpha,\beta/\eta\rbrace$, then  choice  $t_{k+1}\in (t_k,  \min\lbrace T,t_k+\beta\rbrace]$ such that $\eta a+(\eta+1) \ell (t_{k+1}-t_k)\leq\beta$ and $a+\ell(t_{k+1}-t_k)\leq \alpha$. By Theorem \ref{thmSemismoothMethod}, with $t:=t_k$, $s:=t_{k+1},$ and $\bx:=x_k$, there is $x_{k+1}\in\mathbb{R}^n$, generated by Algorithm \ref{algPathFollowing}, such that \eqref{eqErrorGridEstimate} holds.

 If $\Vert x_{k}-x(t_k)\Vert \geq\min\lbrace \alpha,\beta/\eta\rbrace$, compute $u_{i+1}\in\mathbb{R}^n$ by the iterate \eqref{eqIteration} with the initial point $ u_0:=x_{k}$, such that $\Vert u_{i+1}-x(t_k)\Vert <\min\lbrace \alpha,\beta/\eta\rbrace$.  Redefine $x_k:=u_{i+1}$ and apply again Theorem \ref{thmSemismoothMethod} with $t:=t_k$, $s:=t_{k+1}$, and $\bx:=x_k$.
\end{proof}

Strong metric subregularity around each point of continuous path guarantees nonemptyness of the set $\mathcal{Q}_{F, x(t)}^{\eta, \kappa}(t,u)$ for each $t\in[0,T]$ and each $u\in \ball_{\mathbb{R}^n}[x(t),\alpha]$ for some $\alpha>0$.
\begin{proposition}\label{propExistenceMatrices}  Consider a set-valued mapping $G: \mathbb{R}^n\tto \mathbb{R}^n$, with a closed graph, a continuously differentiable single-valued mapping $p: \mathbb{R}\longrightarrow \mathbb{R}^n$ with Lipschitz constant $\ell$, positive constants  $T, \kappa, \eta$, and $\alpha$ and a mapping $x:[0,T] \longrightarrow \mathbb{R}^n$ satisfying  \eqref{eqInclusionParametric} with $F:=G-p$. Suppose that for each $t\in [0, T]$ the mapping $G$ is strongly metrically subregular around $(x(t),p(t))$.
	
	Then  there are $\kappa>0$ and $\alpha>0$ such that for each $t\in [0,T]$  we have $\mathcal{Q}_{G-p, x(t)}^{\eta, \sqrt{n(1+\ell^2\kappa^2+\kappa^2)}}(t,u)\neq \emptyset$ for  each $u \in \ball_{\mathbb{R}^n}[x(t),\alpha]$ with $u\notin S(t)$. 
\end{proposition}
\begin{proof} Theorem \ref{thmStabilitySubregularity}, with $F:=G$ and $\varphi:=x$, implies that there are $\kappa>0$ and $\alpha>0$ such that for each $t\in [0,T]$ the mapping $G$ is strongly metrically subregular around $(x(t),p(t))$ with the constant $\kappa$ and \nns $\ball_{\mathbb{R}^n}[x(t),\eta\alpha]$ and $\ball_{\mathbb{R}^n}[p(t),\eta\alpha]$.
	
	Thus fix any $t\in [0,T]$  and any $u\in  \ball_{\mathbb{R}^n}[x(t),\alpha]$ with  $u \notin S(t)$. Then there is $(\hat{x},\hat{y})\in \gph\,(G-p(t))$ such that $\Vert (\hat{x}-x(t),\hat{y})\Vert\leq \eta\Vert u-x(t)\Vert\leq \eta\alpha$. Since $G$ is strongly metrically regular at $(\hat{x},\hat{y})$ with the constant $\kappa$, by Proposition \ref{propMatricesExistence}, with $F:=G$,   there is a matrix $B \in \mathbb{R}^{n \times n}$ with $\|B\| \leq \kappa$ such that $(I_n, B) \in \mathcal{A}_{\text{reg}}\, G(\hat{x}, \hat{y})$. Thus $\Vert(I_n\mid \dot{p}(t) \mid B)\Vert_F\leq \sqrt{n(1+\ell^2 \kappa^2+\kappa^2)}$, therefore $\mathcal{Q}_{F,x(t)}^{\eta, \sqrt{n(1+\ell^2\kappa^2+\kappa^2)}}(t,u)\neq \emptyset$.
\end{proof}
\section{Numerical Implementation of Semismooth$^*$ Path-Following Methods}

In this section, we present a numerical implementation of semismooth$^*$ path-following methods introduced in this paper for  elementary problems arising in electric circuits. We consider a model represented by parametric generalized equations
\begin{equation}	\label{eqAplInclusion}
	0 \in g( x(t))-p(t) + F(x(t)) \quad \text{for} \quad t \in [0, T],
\end{equation}
where $ g: \mathbb{R}^n \longrightarrow \mathbb{R}^n$ is continuously differentiable single-valued mapping,  $p: [0,T]\longrightarrow\mathbb{R}^n$ is a single-valued mapping, and $ F: \mathbb{R}^n \rightrightarrows \mathbb{R}^n $ is a set-valued mapping with a closed graph.

For a fixed $ t \in [0, T] $, we reformulate the problem, following the framework of \cite{GOV2022}, as

\begin{equation*}
	p(t) \in H(x(t), d(t)) := \left( \begin{array}{c} g(x(t)) + F(d(t)) \\ x(t) - d(t) \end{array} \right),
\end{equation*}

where $ H: \mathbb{R}^{2n} \rightrightarrows \mathbb{R}^{2n} $.

This reformulation simplifies the computation of the approximation step in the algorithms. We first need to define the selection of a set-valued mapping.

We say that $ F:\mathbb{R}^n \tto \mathbb{R}^n $ has a \emph{selection} around $ \bar{x} $ for $ \bar{y} $ if there exists a single-valued mapping $ s $ defined on a \nn $ U $ of $ \bar{x} $ such that $ s(\bar{x}) = \bar{y} $ and $ s(x) \in F(x) $ for every $ x \in U $.

\begin{proposition}
	\label{propApprox} 
	Let $ x(\cdot) $ be a solution of \eqref{eqAplInclusion} and assume that there exists some $ \lambda > 0 $ such that for each $ t \in [0, T] $, the mapping $ (I_n + \lambda F)^{-1} $ has a single-valued Lipschitz continuous selection $ s_t $ around $ x(t) - \lambda g(x(t))+\lambda p(t) $ for $ x(t) $. Then for each $ t \in [0, T] $, there exist $ \delta > 0 $ and $ \eta > 0 $ such that if $ x \in \ball_{\mathbb{R}^n}[x(t), \delta] $, the vectors $ \hat{d} := s_t(x - \lambda g(x)+\lambda p(t)) $ and $ \hat{y} := \left( \frac{1}{\lambda}(x - \hat{d}), x - \hat{d} \right) \in H(x, \hat{d}) $ satisfy the estimate
	
	\begin{equation*}
		\| (x, \hat{d}, \hat{y}) - (x(t), x(t), 0) \| \leq \eta \| x - x(t) \|.
	\end{equation*}
\end{proposition}

\begin{proof} 
	We follow the proof from \cite[Proposition 5.1]{GMOV2023}. Fix any $ t \in [0, T] $ and find $ \delta > 0 $ such that $ \ell_s > 0 $ is a Lipschitz constant of $ s $ on $ \ball_{\mathbb{R}^n}[x(t) - \lambda g(x(t))+\lambda p(t), (1 + \lambda \ell_g)\delta] $ and $ \ell_g $ is a Lipschitz constant of $ g $ on $ \ball_{\mathbb{R}^n}[x(t), \delta] $. Fix any $ x \in \ball_{\mathbb{R}^n}[\bar{x}, \delta] $. Then $ x - \lambda g(x)+\lambda p(t) \in \ball_{\mathbb{R}^n}[x(t) - \lambda g( x(t))-p(t), (1 + \lambda \ell_g)\delta] $. Let $ \hat{d} := s_t(x - \lambda g( x)+\lambda p(t)) $. By the definition of $ s_t $, we have
	\begin{eqnarray*}
		\| \hat{d} - x(t) \| & = \| s_t(x - \lambda g( x)+\lambda p(t)) - s_t(x(t) - \lambda g(x(t))+p(t)) \| \\ 
		& \leq \ell_s \| x - \lambda g(x)+\lambda p(t) - (x(t) - \lambda g(x(t))+\lambda p(t)) \| \\ 
		& \leq \ell_s (1 + \lambda \ell_g) \| x - x(t) \|.
	\end{eqnarray*}
	
	Moreover, since $ x - \lambda f( x)+\lambda p(t) \in \hat{d} + \lambda F(\hat{d}) $, it follows that $ \hat{y} \in \gph \, H(x, \hat{d}) $. Additionally, we have
	$$
	\| \hat{y} \| \leq \left( 1 + \tfrac{1}{\lambda} \right) \| \hat{d} - x(t) \| \leq \left( 1 + \frac{1}{\lambda} \right) \ell_s (1 + \lambda \ell_g) \| x - x(t) \|.
	$$
	
	Thus, we obtain
	\begin{equation*}
		\| (x, \hat{d}, \hat{y}) - (x(t), x(t), 0) \| \leq \max\left\{ 1, \left( 1 + \tfrac{1}{\lambda} \right) \ell_{s} (1 + \lambda \ell_g), \ell_{s} (1 + \lambda \ell_g) \right\} \| x - x(t) \|.
	\end{equation*}
	
\end{proof}

It is a well-known fact that for each positive $ \lambda $, the mapping $ (I+\lambda G)^{-1} $ is single-valued and Lipschitz continuous on $ \mathbb{R}^n $ with a constant of 1 for any set-valued mapping $ G:\mathbb{R}^n \tto \mathbb{R}^n $ that is maximal monotone. A set-valued mapping $ G : \mathbb{R}^n \rightrightarrows \mathbb{R}^n $ is said to be \emph{monotone} if
\begin{equation*}
	\langle y - v, x - u \rangle \geq 0 \quad \text{for every} \quad (x, y), (u, v) \in \text{Graph} \, G
\end{equation*}

and $G $ is said to be \emph{maximal monotone} if it is monotone and there exists no other monotone mapping whose graph strictly contains the graph of $ G $.

\begin{example1} 
	For each $ i=1,2,\ldots, N $, consider a set-valued mapping $ F_i:\mathbb{R}^n \tto\mathbb{R}^n $ that is maximal monotone. Define a set-valued mapping $ G:\mathbb{R}^n\tto\mathbb{R}^n $ by
	
	\begin{equation*}
		\gph\,G:=\bigcup_{i=1}^k \gph\,G_i.
	\end{equation*}
	
	Then the mapping $ (I_n + G)^{-1} $ has at least one selection around $ \bx $ and $ \by $ for each $ (\bx,\by) \in \gph\,G $, which is Lipschitz continuous on the whole $ \mathbb{R}^n $ with a constant of 1.
\end{example1}

Using Remark \ref{remCodevSingle} and Remark \ref{remCodevInv}, we get the following lemma.
\begin{lemma}
	\label{lemCoderivativeHt}
 Let $ x:[0,T]\longrightarrow \mathbb{R}^n $ be a solution of \eqref{eqAplInclusion} and assume that there is some $ \lambda > 0 $ such that for each $t\in [0,T]$, the mapping  $ (I_n + \lambda F)^{-1} $ has a single-valued Lipschitz continuous localization $s_t$ around $x(t) - \lambda g(x(t))+p(t)+1/\lambda\,x(t) $ for $x(t)$. 
For each $t\in [0,T]$ and each $((x,d), (v, w))\in \gph\,(H-p(t))$, we have
\begin{eqnarray*}
	D^*H((x,d), (v,w))=
	\left( \begin{array}{ccccc}
		\nabla g( x)^T& &-1/\lambda\, I_n+ 1/\lambda D^*(I_n+\lambda F)(d,v-g( x)+p(t)+1/\lambda\,d)\\
		I_n& &-I_n	\end{array}
	\right),
\end{eqnarray*}
where 
\begin{eqnarray*}
 \lbrace (A^Tu,u): u\in \mathbb{R}^n\text{ and } A\in \partial_B s_t(v-g( x)+p(t)+1/\lambda\,d) \rbrace\subset\gph\,	D^*(I_n+\lambda F)(d,v-g(x)+p(t)+1/\lambda d).
\end{eqnarray*}
\end{lemma}
\begin{lemma}Suppose that the assumptions of Lemma \ref{lemCoderivativeHt} are satisfied.
For each $t\in [0,T]$, each $((x,d), (v, w))\in \gph\,(H-p(t))$, and each  $A, B, C, D \in \mathbb{R}^{n\times n}$ and for each $J\in \partial_B s_t(v-g(x)+p(t)+1/\lambda\,d)$ define the matrices
	\begin{eqnarray*}
	M:=\left( \begin{array}{cccc}
		A \nabla g( x)  -1/\lambda \, C J +1/\lambda\,C&  & A -C J\\
		B \nabla g( x) -1/\lambda \, D J +1/\lambda\,D& &B -DJ		\end{array}
	\right) \text{ and }
	N:= \left( \begin{array}{ccc}
		A &  & 	C J\\
	B	& & D J		\end{array}	\right).
	\end{eqnarray*}
If $M$	is non-singular, we have
\begin{eqnarray}
	\label{eqAregHt}
		\left( 
		M,  N \right)\in \mathcal{A}_{\text{reg}} H((x,d),(v,w)).
\end{eqnarray}

\end{lemma}
\begin{proof}Fix any $t\in [0,T]$, any $((x,d), (v, w))\in \gph\,H$, and
  any $A, B, C, D \in \mathbb{R}^{n\times n}$, and  any  $J\in \partial_B s_t(v-g(x)+p(t)+1/\lambda\,d)$. Further, by Lemma \ref{lemCoderivativeHt}, for each $i \in \lbrace 1, 2,\dots,n\rbrace$ we get
\begin{eqnarray*}
	\left( \left( \begin{array}{ccc}
		(A)_i^T \\
		J^T(C)_i^T 		\end{array}	\right), \left( \begin{array}{ccc}
		\nabla g( x)^T (A)_i^T  -1/\lambda \, J^T(C)_i^T +1/\lambda\,(C)_i^T	\\
		(A)_i^T -J^T(C)_i^T	\end{array}
		\right)\right)\in \gph D^*H((x,d), (v,w))
\end{eqnarray*}
 and the same is satisfied for $(B, D)$ instead of $(A, C)$. If $M$ is non-singular, then \eqref{eqAregHt}  holds.
\end{proof}

We suppose that $F$ is maximal monotone and $g$ is  strongly monotone, i.e., we say that $g$ is \emph{strongly monotone} if there exists $c > 0$ such that 
\begin{eqnarray*}
	\left\langle g(x) - g(u), x - u \right\rangle \geq c \Vert x - u \Vert^2 \quad \text{for each} \quad x,u \in \mathbb{R}^n.
\end{eqnarray*}
By \cite[12.54 Proposition]{RW1998}, the mapping $(g + F)^{-1}$ is Lipschitz continuous on the whole of $\mathbb{R}^n$ with the constant $1/c$.
This guarantees that the mapping $g + F$ is strongly metrically subregular around each point on its graph with the constant $1/c$.

Furthermore, we assume that $p$ is Lipschitz continuous on $[0,T]$ with a constant $\ell$.
Then, the unique (exact) solution $x(\cdot)$ to \eqref{eqAplInclusion} is given by
$$
x(t) = (g + F)^{-1}(p(t)) \quad \text{for each} \quad t \in [0,T],
$$
and is Lipschitz continuous on $[0,T]$ with the constant $\ell/c$.
Under the these assumptions the mapping $H^{-1}$ is also Lipschitz continuous on whole $\mathbb{R}^n$ and for each $t\in [0,T]$ we have
\begin{eqnarray*}
	\begin{pmatrix}
		x(t)\\
		x(t)
	\end{pmatrix}= H^{-1}((p(t),0)).
\end{eqnarray*}

For non-uniform semismooth$^*$ path-following  method, we consider parametric  forward-backward splitting method, where the corresponding mapping $\mathcal{T}:\mathbb{R}\times \mathbb{R}^n\times\mathbb{R}^n \longrightarrow \mathbb{R}^{2n}$ is defined by 
$$
\mathcal{T}(t,x,d):=\left( \left(I_n+\lambda F\right)^{-1}\big(x-\lambda g(x)+\lambda p(t)\big),d-\lambda(d-x)\right)\quad\text{for}\quad (t,x,d)\in\mathbb{R}\times \mathbb{R}^n\times\mathbb{R}^n,
$$
 where $\lambda>0$. 
Under the assumptions on $F$ and $g$, the mapping $\mathcal{T}$ is well defined and is a contraction mapping for Lipschitz continuous $g$ on whole $\mathbb{R}^n$ with a constant $\ell$  and for  $\lambda\in (0,2c/\ell^2)$ with $\lambda<1$.

In this section, we examine specific examples of electric circuits, as discussed in \cite{CR2018}. For our analysis, we apply different algorithms tailored to the unique requirements of each example.

For Example 7.2 (i) and Example 7.3, we utilize Algorithm \ref{algPathFollowingUniform} for the mapping $H$. This algorithm is particularly effective for scenarios, where a uniform grid is employed. To ensure the convergence of this algorithm, it is crucial to choose a sufficiently small step size $ h $ for the grid. The conditions for this convergence are rigorously established in Theorem \ref{algPathFollowingUniform}, which outlines how the step size affects the accuracy of the solution.

On the other hand, for Example 7.2 (ii), we implement Algorithm \ref{algSemismoothWholeT} for the mapping $ H $. This algorithm operates under different constraints and requires a more adaptive approach. Specifically, at each time step $ t_k $, it is necessary to determine the next time step $ t_{k+1} $ and, if needed, refine the solution $ x_k $ using the mapping $ \mathcal{T} $.

The calculations are based on the following idea: Theorem \ref{thmStabilitySubreg}(iv) implies that there exist constants $ \kappa > 0 $ and $ \alpha > 0 $ such that for each $ t \in [0, T] $ and each $ (x, d) \in \ball_{\mathbb{R}^{2n}}[(x(t), x(t)), \alpha] $, we have
$$
\| (x, d) - (x(t), x(t)) \| \leq \kappa \, \dist(0, H(x, d)-p(t)).
$$
Thus, if for some $ k \in \mathbb{N} $, the (Newton) iteration $ (x_{k+1}, d_{k+1}) \in \ball_{\mathbb{R}^{2n}}[(x(t_{k+1}), x(t_{k+1})), \alpha] $ and for some $ y_{k+1} \in H(x_{k+1}, d_{k+1}) -p(t_{k+1})$, we have
$$
\kappa \| y_{k+1} \| \leq \frac{1}{2} \max \left\{ \| (x_{k}, d_k) - (x(t_k), x(t_k)) \|, t_{k+1} - t_k \right\} + \hat{\ell} (t_{k+1} - t_k),
$$
then
$$
\| (x_{k+1}, d_{k+1}) - (x(t_{k+1}), x(t_{k+1})) \| \leq \frac{1}{2} \max \left\{ \| (x_{k}, d_k) - (x(t_k), x(t_k)) \|, t_{k+1} - t_k \right\} + \ell (t_{k+1} - t_k),
$$
where $ \hat{\ell}:=\sqrt{2}\,\ell $ is the Lipschitz constant of $ (x(\cdot), x(\cdot)) $.

Since the constant $ \alpha $ is difficult to find, in the numerical implementation, we search for $ t_{k+1} $ and refine $ x_k $ such that the corresponding $ x_{k+1} $ satisfies
\begin{eqnarray}
\label{eqEstimated2}
\kappa \| y_{k+1} \| \leq \frac{1}{2} \max \left\{ c_k, t_{k+1} - t_k \right\} + \ell (t_{k+1} - t_k),
\end{eqnarray}
where $ \| (x_{k}, d_k) - (x(t_k), x(t_k)) \| \leq c_k := \frac{1}{2} \max \left\{ c_{k-1}, t_{k} - t_{k-1} \right\} + \ell (t_{k} - t_{k-1}) $ and $ c_0 := 0 $.

To find $t_{k+1}$, we choose a maximum step size $h_{\text{max}} > 0$ , and for some $a \in (0, 1)$ (both independent of $k$), we determine $i \in \mathbb{N}_0$ such that $t_{k+1} := t_k + a^i h_{\text{max}}$, and the corresponding $(x_{k+1}, d_{k+1})$ satisfies the condition \eqref{eqEstimated2}. If no such $i$ exists, we refine $(x_k, d_k)$ using the iterations by the mapping $\mathcal{T}$, and repeat the process until a suitable $i$ is found.

\begin{example1} \label{ex_simple_circuit}
	Consider a simple circuit with a non-linear resistor defined by the function $g: \mathbb{R} \longrightarrow \mathbb{R}$, a source $E > 0$, and an input signal $u: [0, T] \longrightarrow \mathbb{R}$. The instantaneous current is denoted by $i$, and the diode is given by a set-valued mapping $F: \mathbb{R} \rightrightarrows \mathbb{R}$.	
	By Kirchhoff voltage law, we have the relation \eqref{eqAplInclusion},	where $p = u - E$ and $x = i$. The term $g(x(t))$ represents the voltage across the resistor ($V_R$) and $F(x(t))$ represents the voltage across the diode ($V_D$). 
We consider two cases:
	
	\begin{enumerate}
		\item[\rm (i)] in the case of a practical diode, the voltage signal is given by $p(t) = 3 \sin(t)$,  the voltage of the resistor  is described by $g(x) = \sinh(x)$, and characteristic of the diode is defined as
		$$
		F(x) = 
		\begin{cases} 
			[V_1, V_2], & \text{if } x = 0, \\
			V_2 + \sqrt{x}, & \text{if } x > 0, \\
			V_1 - \sqrt{-x}, & \text{if } x < 0,
		\end{cases}
		$$
		where $V_1 < 0 < V_2$. For the simulation, the parameters are set as follows: $V_1 = -2$, $V_2 = 1$, and $T=3$. Simulation results for this setting is illustrated in Figure \ref{fig_sol_rate1};
		
		\item[\rm (ii)] for an ideal diode, the input signal is given by $p(t) = \sin(2\pi t)$,  voltage of the resistor  is given by $g(x) = \log(x+\sqrt{x^2+1})$, and characteristic of the diode is described by the normal cone to $[0,\infty)$, i.e., $F(x) = N_{\mathbb{R}_+}(x)$ for $x\in\mathbb{R}$. The results of the simulation (for $T=3$) are illustrated in Figure \ref{figExample1}.
	\end{enumerate}
\end{example1}
\begin{example1}
	\label{ex_2D_circuit}
	Consider the circuit, which includes resistances $R_B > 0$ and $R_L > 0$, two input signals $u_1$ and $u_2$, and a P-N-P transistor. The operation of the transistor follows the Ebers-Moll model \cite[p. 409]{SedraSmith2004}, involving two diodes in opposite directions and two current-controlled sources $\alpha I_C$ and $\alpha_N I_E$. Here, $\alpha_N \in [0, 1]$ represents the current gain in normal operation, and $\alpha_I \in [0, 1]$ denotes the inverted gain. The emitter and collector currents, $i_E$ and $i_C$, are defined by
	$$
	i_E = I_E - \alpha_I I_C, \quad i_C = I_C - \alpha_N I_E.
	$$
	This system can be rewritten in matrix form as:
	$$
	\begin{pmatrix}
		i_E \\ 
		i_C
	\end{pmatrix}
	=
	\begin{pmatrix}
		1 & -\alpha_I \\ 
		-\alpha_N & 1
	\end{pmatrix}
	\begin{pmatrix}
		I_E \\ 
		I_C
	\end{pmatrix}.
	$$
Using Kirchhoff laws, the base current $i_B$ is determined as:
	$$
	i_B = -(i_E + i_C),
	$$
	and the corresponding voltage equations are
\begin{eqnarray*}
&&0=R_B (-i_E - i_C) +u_1 - V_E  ,\\
&&0=  R_B (i_E + i_C)+V_C + u_2 + R_L i_C - u_1.
\end{eqnarray*}
	In this model, the diode characteristics are described by set-valued mappings
	$$
	F_1(x) =
	\begin{cases}
		[V_{E_1}, V_{E_2}], & \text{if } x = 0, \\
		V_{E_1} + \sqrt{x}, & \text{if } x > 0, \\
		V_{E_2} -  \sqrt{-x}, & \text{if } x < 0,
	\end{cases}
	\quad\text{and}\quad\quad F_2(x) =
	\begin{cases}
		[V_{C_1}, V_{C_2}], & \text{if } x = 0, \\
		V_{C_1} + \sqrt{x}, & \text{if } x > 0, \\
		V_{C_2} - \sqrt{-x}, & \text{if } x < 0,
	\end{cases}
	$$
	where the constants  $V_{E_2}>0>V_{E_1}$ and $V_{C_2}>0>V_{C_1}$. 
	Thus, we obtain $V_E \in F_1(I_E)$ and $V_C \in F_2(I_C)$. The overall voltage relationships are expressed as \eqref{eqAplInclusion} with $x:=(x_1,x_2)^T:=(I_E,I_C)^T$,
	\begin{eqnarray*}
		p:=\begin{pmatrix}
			u_1\\
			u_1-u_2
		\end{pmatrix}, \quad
		g(x_1,x_2):=
		\begin{pmatrix}
			(1-\alpha_N)R_B & (1-\alpha_I)R_B \\ 
			(1-\alpha_N)R_B-\alpha_N R_L & (1-\alpha_N)R_B + R_L
		\end{pmatrix}
		\begin{pmatrix}
			x_1 \\ 
			x_2
		\end{pmatrix},
	\end{eqnarray*}
	and 
	\begin{eqnarray*}
		F(x_1,x_2):= (F_1(x_1), F_2(x_2))^T.
	\end{eqnarray*}
	
For the simulation, the parameters are set as follows: $V_{E_1} = -3$, $V_{E_2} = 1$, $V_{C_1} = -1$, $V_{C_2} = 2$, $R_B = 1$, $R_L = 1$, $\alpha_I=0.1$,  $\alpha_N=0.5$, and $T=2\pi$. The input signals are defined as $u_1 = \sin{(t)}$ and $u_2 = 10 \sin{(t)}$. Figure \ref{fig_sol_con2} illustrates the simulation results for these settings.

\begin{figure}
\vspace{-0.5cm}
\centering
\begin{minipage}[c]{0.45\textwidth}
\includegraphics[width=\textwidth]{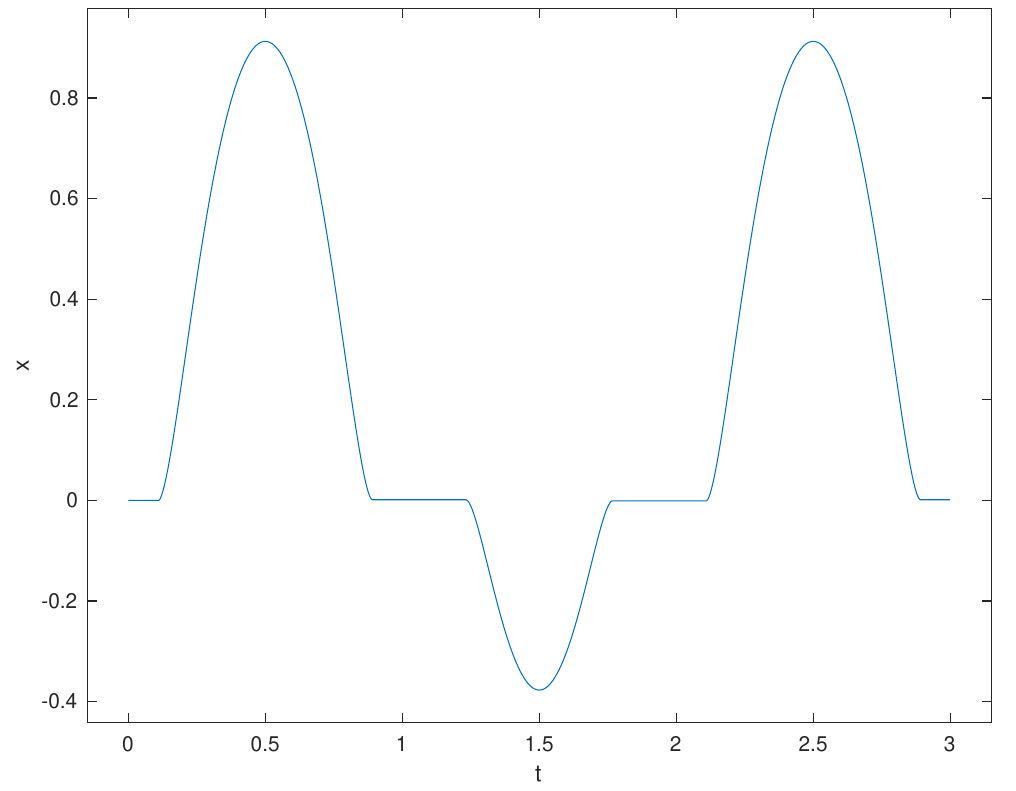}
\end{minipage} 
\begin{minipage}[c]{0.45\textwidth}
\includegraphics[width=\textwidth]{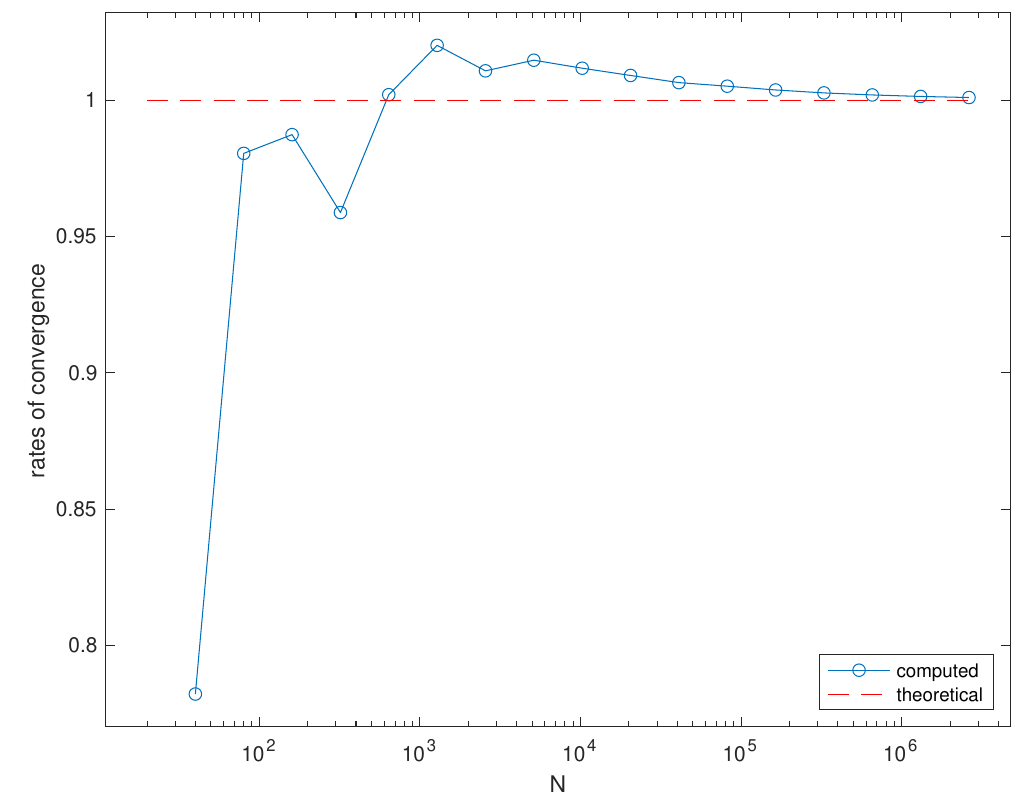}
\end{minipage}
\caption{Example \ref{ex_simple_circuit} (i) - numerical solution (left) obtained by Algorithm \ref{algPathFollowingUniform} and observed rates of convergence (right).}
\label{fig_sol_rate1}
\vspace{0.5cm}
\centering
\begin{minipage}[c]{0.45\textwidth}
\includegraphics[width=\textwidth]{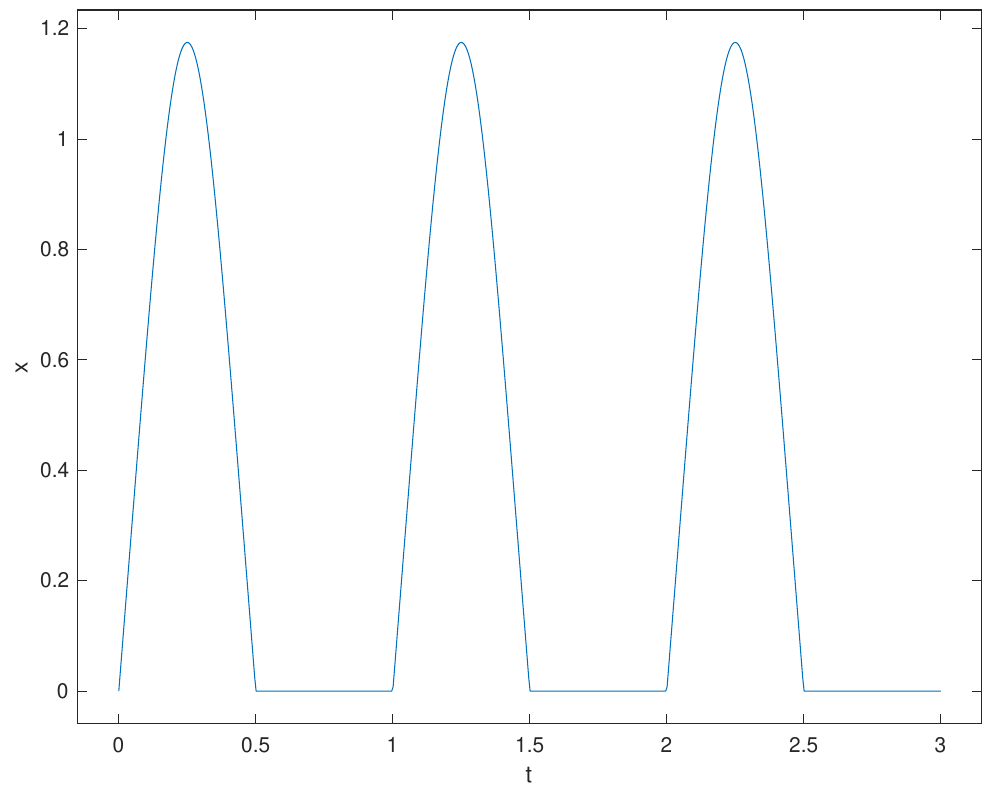}
\end{minipage} 
\begin{minipage}[c]{0.45\textwidth}
\includegraphics[width=\textwidth,height=0.8\textwidth]{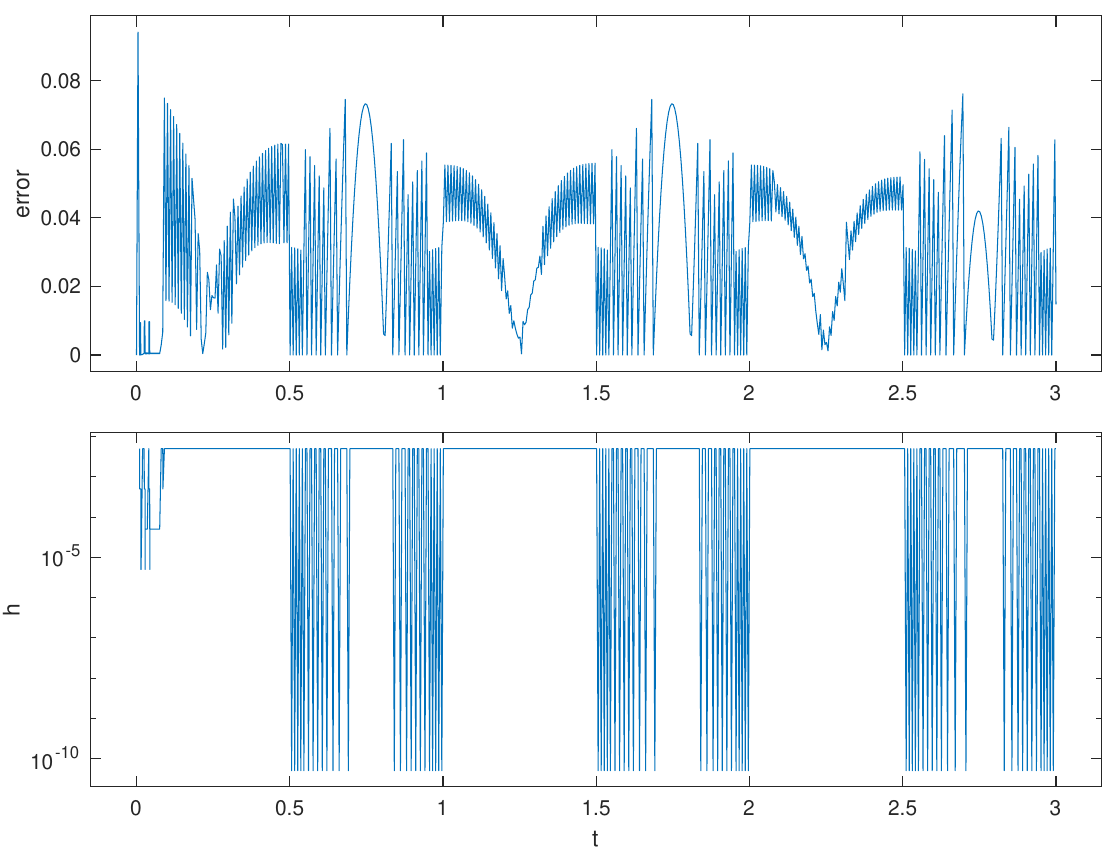}
\end{minipage}
\caption{Example \ref{ex_simple_circuit} (ii) - (exact) solution (left)   and observed error (top right) and length of steps  (bottom right) obtained by Algorithm \ref{algSemismoothWholeT}.} 
\label{figExample1}
\vspace{0.5cm}
\centering
\begin{minipage}[c]{0.45\textwidth}
\includegraphics[width=\textwidth]{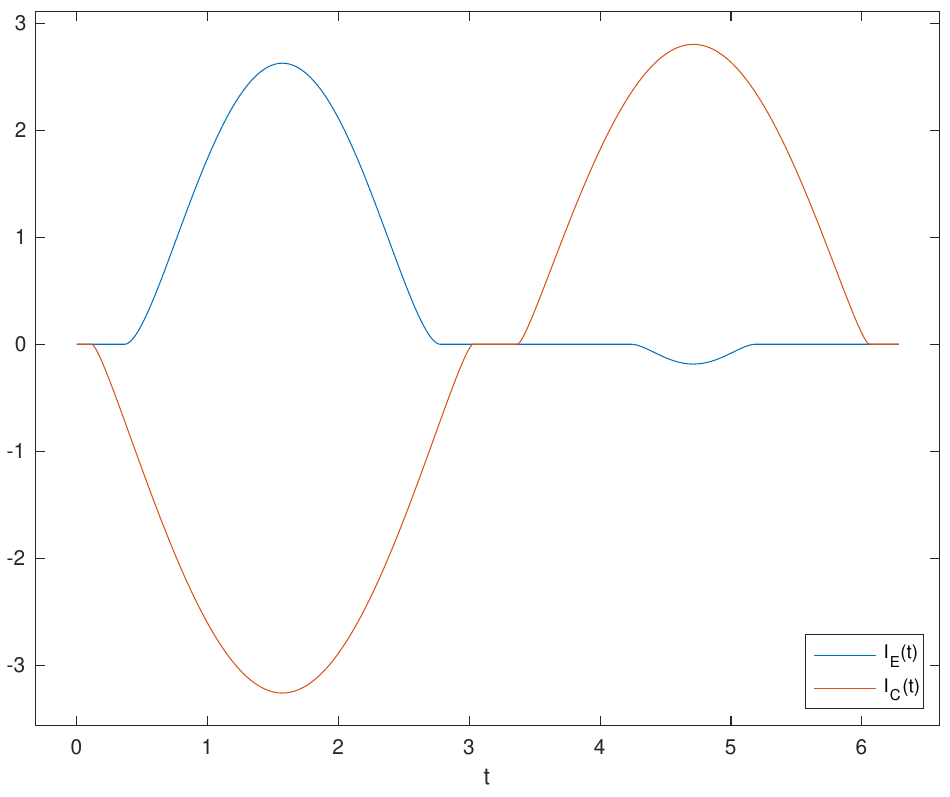}
\end{minipage} 
\begin{minipage}[c]{0.45\textwidth}
\includegraphics[width=\textwidth,height=0.84\textwidth]{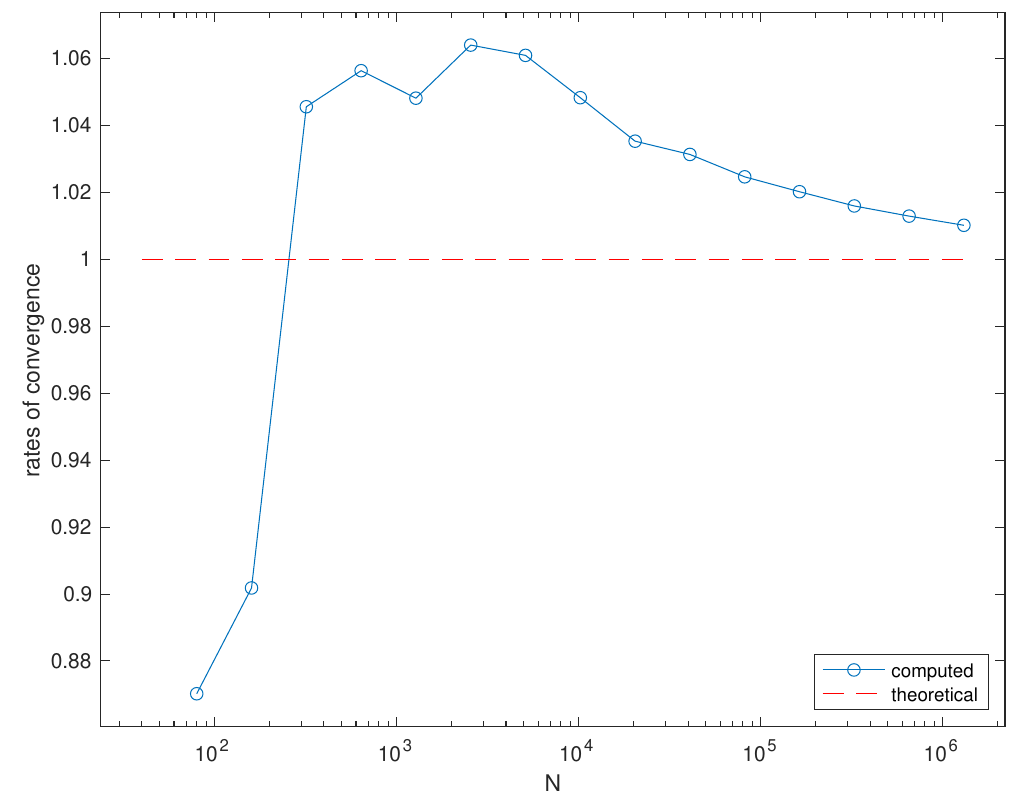}
\end{minipage}
\caption{Example \ref{ex_2D_circuit} - numerical solution (left) obtained by Algorithm \ref{algPathFollowingUniform} and observed rates of convergence (right).}
\label{fig_sol_con2}
\end{figure}
\end{example1}
Let us note that, in Examples \ref{ex_2D_circuit} and \ref{ex_simple_circuit} (ii), we used the exact  solution as our reference solution. For Example \ref{ex_simple_circuit} (i), we employed a highly accurate numerical solution obtained through the 'fzero' library in MATLAB as the reference solution.

\section{Conclusion}

In this paper, we introduced the new property called uniform semismoothness$^*$ and provided sufficient conditions for it. Additionally, we established conditions for uniform strong metric subregularity around the reference point, utilizing the stability of strong metric subregularity under Lipschitz continuous single-valued perturbations. Finally, we presented two numerical path-following methods for solving parametric inclusions.

The first method, which assumes that the corresponding graph of the set-valued mapping is uniformly semismooth$^*$ along the reference solution, guarantees linear convergence of the grid error on a uniform grid with decreasing step sizes, a result confirmed by our numerical experiments. The second method, which relies on point-wise semismoothness$^*$ along the reference solution, offers a good estimate for the grid error in a more general setting. However, its numerical implementation is challenging, necessitating the use of heuristics with adaptive step sizes to achieve practical solutions.

Future research will focus on improving the implementation of this second method, enhancing its applicability and performance in diverse scenarios.


\end{document}